\documentclass[a4paper,draft,11pt]{article}

\usepackage{pb-diagram,amsmath,amssymb}

\usepackage{amsthm}
\usepackage[T2A]{fontenc}
\usepackage[cp1251]{inputenc}
\usepackage[active]{srcltx}

\usepackage{amsthm}

\newtheorem{theorem}{Theorem}
\newtheorem{definition}[theorem]{Definition}
\newtheorem{lemma}[theorem]{Lemma}
\newtheorem{lemma-definition}[theorem]{Lemma-Definition}
\newtheorem{proposition-definition}[theorem]{Proposition-Definition}
\newtheorem{proposition}[theorem]{Proposition}
\newtheorem{corollary}[theorem]{Corollary}

\newcommand{\CC}{\mathbb C}

\newcommand{\FF}{\mathbb F}

\newcommand{\NN}{\mathbb N}

\newcommand{\PP}{\mathbb P}
\newcommand{\QQ}{\mathbb Q}
\newcommand{\RR}{\mathbb R}

\newcommand{\ZZ}{\mathbb Z}
\newcommand{\cA}{\mathcal A}

\newcommand{\cH}{\mathcal H}

\def\myarrow{\ \hbox to 2em{\leaders
\hbox to 0.5ex{\hss\raise 0.55ex\hbox to 0.3ex{\hrulefill}\hss}
\hfill\,\llap{$>$}}\ }

\title{  Riemann Hypothesis Analogue for locally finite modules over the absolute Galois group of a finite field  \footnote{
Supported by   Contract 144/2015 and  Contract 57/12.04.2016  with the Scientific Foundation of Kliment Ohridski  University of Sofia. } }
\author{Azniv Kasparian, Ivan Marinov
}
\date{      }

\begin{document}
\maketitle

\thispagestyle{empty}
%%%%%%%%%%%%%%%%%%%%%%%%%%%%%%%%%%%%%%%%%%%%%%%%%%%%%%%%%%%%%%%%%%%%%%%%%%%%%%%%%%%%%%%%%%%%%%%%%%%%%%%%%%%%%%%%%%%%%%%%

\begin{abstract}
 The  article  provides a sufficient condition for a locally finite module $M$ over the absolute Galois group $\mathfrak{G} = {\rm Gal} ( \overline{\FF_q} / \FF_q)$ of a finite field $\FF_q$ to satisfy the Riemann Hypothesis Analogue with respect to the projective line $\PP^1( \overline{\FF_q})$.
The condition holds for all smooth irreducible projective curves   of positive genus, defined over $\FF_q$.
 By construction of an explicit example we establish that the scope of our main result is larger than the class of the smooth irreducible projective varieties, defined over $\FF_q$.
\end{abstract}

%%%%%%%%%%%%%%%%%%%%%%%%%%%%%%%%%%%%%%%%%%%%%%%%%%%%%%%%%%%%%%%%%%%%%%%%%%%%%%%%%%%%%%%%%%%%%%%%%%%%%%%%%%%%%%%%%%%%%%%%%%
\section{Introduction}
%%%%%%%%%%%%%%%%%%%%%%%%%%%%%%%%%%%%%%%%%%%%%%%%%%%%%%%%%%%%%%%%%%%%%%%%%%%%%%%%%%%%%%%%%%%%%%%%%%%%%%%%%%%%%%%%%%%%%%%%%%%%

 The absolute Galois group of a finite field $\FF_q$ is isomorphic to the profinite completion
 $$
 \widehat{\ZZ} = \{ ( l_m ( {\rm mod} \, m)) _{m \in \NN} \in \prod\limits _{m=1} ^{\infty} ( \ZZ_m, +) \, \vert \, l_n \equiv l_m ( {\rm mod} \, m) \, \mbox{  for  } \, \forall m / n \}
 $$
 of the infinite cyclic group $(\ZZ, +) = \langle \Phi _q \rangle$, generated by the Frobenius automorphism
  $\Phi _q : \overline{\FF_q} \rightarrow \overline{\FF_q}$, $\Phi _q (a) = a^q$ for $\forall a \in \overline{\FF_q}$.
  The diagonal map
  $$
  \Delta : \ZZ \longrightarrow \widehat{\ZZ},  \quad
  \Delta (z) = (z({\rm mod} \, m)) _{m \in \NN} \ \ \mbox{  for } \ \ \forall z \in \ZZ
  $$
with a dense image  is an embedding,  so that  $\ZZ$ is residually finite.

  A set $M$ with an action of a group $G$ will be called a $G$-module.

  \begin{definition}    \label{LocallyFiniteModule}
  A $\mathfrak{G} = {\rm Gal} ( \overline{\FF_q} / \FF_q)$-module  $M$ is locally finite if all $\mathfrak{G}$-orbits on $M$ are finite and for any $n \in \NN$ there are at most finitely many $\mathfrak{G}$-orbits on $M$ of cardinality $n$.

  The degree of a finite $\mathfrak{G}$-orbit ${\rm Orb} _{\mathfrak{G}} (x) \subset M$ is its cardinality,
  $
  \deg {\rm Orb} _{\mathfrak{G}} (x) := \left| {\rm Orb} _{\mathfrak{G}} (x) \right|.
  $
  \end{definition}

 The smooth irreducible projective curves $X / \FF_q \subset \PP^n ( \overline{\FF_q})$, defined over a   $\FF_q$ are examples of locally finite $\mathfrak{G} = {\rm Gal} ( \overline{\FF_q} / \FF_q)$-modules.

\begin{definition}   \label{ZetaFunction}
If $M$ is a locally finite $\mathfrak{G} = {\rm Gal} ( \overline{\FF_q} / \FF_q)$-module then the formal power series
$$
\zeta _M(t) := \prod\limits _{\nu \in {\rm Orb} _{\mathfrak{G}} (M)} \left( \frac{1}{1 - t^{\deg \nu}} \right) \in \CC [[t]]
$$
is called the $\zeta$-function of $M$.
\end{definition}

In the case of a smooth irreducible curve $X / \FF_q \subset \PP^n ( \overline{\FF_q})$, the $\zeta$-function $\zeta _X(t)$ of $X$ as a locally finite $\mathfrak{G} = {\rm Gal} ( \overline{\FF_q} / \FF_q)$-module coincides with the local Weil $\zeta$-function of $X$.
We fix the projective line $\PP^1 ( \overline{\FF_q})$ as a basic model, to which we compare the locally finite $\mathfrak{G}$-modules $M$ under consideration and recall its $\zeta$-function
$$
\zeta _{\PP^1 ( \overline{\FF_q})} (t) = \frac{1}{(1-t)(1-qt)}.
$$
Note also  that $\zeta _M(0) =1$.

 \begin{definition}   \label{RHA}
 A locally finite $\mathfrak{G} ={\rm Gal} ( \overline{\FF_q} / \FF_q)$-module $M$ satisfies the Riemann Hypothesis Analogue
 with respect to the projective line $\PP^1 ( \overline{\FF_q})$ if
 $$
 P_M(t) = \frac{\zeta _M (t)}{\zeta _{\PP^1 ( \overline{\FF_q})} (t)} =  \sum\limits _{i=0} ^d a_i t^{i} = \prod\limits _{i=1} ^d (1 - \omega _it) \in \CC[t]
 $$
 is a polynomial  with $\left| \omega  _i \right| = \sqrt[d]{\left| \omega  _1 \right| \ldots \left| \omega _d \right|} = \sqrt[d]{\left| a_d \right|}$ for
 $\forall 1 \leq i \leq d$.
 \end{definition}

 In order to explain the etymology of the notion, let us plug in $q^{-s}$, $s \in \CC$ in the $\zeta$-function $\zeta _M(t)$ of $M$ and view
  $$
  \zeta _M \left( q^{-s} \right) = \frac{\prod\limits _{i=1} ^d (1 - \omega _i q^{-s})}{(1 - q^{-s})(1 - q^{1-s})} =
  \frac{\prod\limits _{i=1} ^d (q^s - \omega _i)}{q^{sd-2s+1} (1 - q^s)(1 - q^{s-1})}
  $$
  as a meromorphic function of $s \in \CC$ with poles at $s \in \{ 0,1 \}$.
  The complex zeros $s_o \in \CC$ of $\zeta _M \left( q^{-s} \right)$ have $\omega _i = q^{s_o}= q^{{\rm Re} ( s_o) + i {\rm Im} (s_o)}$, whereas
   $\left| \omega _i \right| = q^{{\rm Re} (s_o)}$ for some  $1 \leq i \leq d$.
Thus, $M$ satisfies the Riemann Hypothesis Analogue with respect to $\PP^1 ( \overline{\FF_q})$ exactly when
 ${\rm Re} (s_o) = \lambda := \log _q \sqrt[d]{\left| a_d \right|} \in \RR ^{\geq 0}$ for any zero $s_o \in \CC$ of $\zeta _M(q^{-s})$.
 All  smooth irreducible curves $X / \FF_q \subset \PP^n ( \overline{\FF_q})$ of genus $g \geq 1$ satisfy the Riemann Hypothesis Analogue with respect to
  $\PP^1 ( \overline{\FF_q})$ by the Hasse-Weil Theorem (cf.\cite{B} or \cite{St}).
More precisely, $P_X(t) = \frac{\zeta _X(t)}{\zeta _{\PP^1 ( \overline{\FF_q})} (t)} = \prod\limits _{i=1} ^{2g} (1 - \omega _it)$
with $\left| \omega _i \right| = q^{\frac{1}{2}}$ for $\forall 1 \leq i \leq 2g$, which is equivalent to ${\rm Re} (s_o) = \frac{1}{2}$ for all the complex zeros $s_o \in \CC$ of $\zeta _X ( q^{-s})$.
That resembles the original Riemann Hypothesis ${\rm Re} (z_o) = \frac{1}{2}$ for the non-trivial zeros $z_o \in \CC \setminus ( -2 \NN)$ of Riemann's $\zeta$-function
$$
\zeta (z) := \sum\limits _{n=1} ^{\infty} \frac{1}{n^z} ,  \quad z \in \CC.
$$

The present article studies the $\zeta$-function $\zeta _M(t)$ of a locally finite $\mathfrak{G} = {\rm Gal} ( \overline{\FF_q} / \FF_q)$-module and provides a sufficient condition for $M$ to satisfy the Riemann Hypothesis Analogue with respect to $\PP^1( \overline{\FF_q})$.
All properties of  $\zeta _M(t)$ under consideration hold for a smooth irreducible projective curve $M = X$, defined over $\FF_q$.
By  its  very definition, $\zeta _M(t)$ is a local $\zeta$-function.
It is interesting to introduce arithmetic objects $A$, whose reductions modulo prime integers $p$ are locally finite ${\rm Gal} ( \overline{\FF_p} / \FF_p)$-modules and to study the global $\zeta$-functions of $A$.
Another topic of interest is the Grothendieck ring of the locally finite $\mathfrak{G} = {\rm Gal} ( \overline{\FF_q} / \FF_q)$-modules and the constructions of motivic $\zeta$-functions.
Let us mention that our $\zeta _M(t)$  is  essentially different from Igusa's $p$-adic $\zeta$-function of a hypersurface and its Archimedean, global or motivic versions.
Our study of the Riemann Hypothesis Analogue for a locally finite $\mathfrak{G} = {\rm Gal} ( \overline{\FF_q} / \FF_q)$-module is motivated by Duursma's notion of a $\zeta$-function $\zeta _C(t)$ of a linear code $C \subset \FF_q ^n$ and the Riemann Hypothesis Analogue for $\zeta _C(t)$, discussed in \cite{D2}.
Recently, $\zeta$-functions have been used for description of the subgroup growth or the representations of a group, as well as of some properties of finite graphs.

We extract some features of Bombieri's proof \cite{B} of the Riemann Hypothesis Analogue for a smooth irreducible curve
$X / \FF_q \subset \PP^n ( \overline{\FF_q})$ of genus $g >0$   and provide a sufficient condition for a locally finite  $\mathfrak{G} ={\rm Gal} ( \overline{\FF_q} / \FF_q)$-module $M$   to satisfy the Riemann Hypothesis Analogue with respect to $\PP^1 ( \overline{\FF_q})$.
The sufficient condition, discussed in our main Theorem \ref{MainTheorem} consists  of three assumptions, which are shown to be satisfied by the smooth irreducible projective curves $X / \FF_q \subset \PP^N ( \overline{\FF_q})$ of genus $g>0$.
 The first  assumption  is the presence of a polynomial $\zeta$-quotient $P_M(t) = \frac{\zeta _M(t)}{\zeta _{\PP^1 ( \overline{\FF_q})} (t)} \in \ZZ[t]$.
  The second  one  is the existence of locally finite $\mathfrak{G} _m = {\rm Gal} ( \overline{\FF_q} / \FF_{q^m})$-submodules $M_o \subseteq M$, $L_o \subseteq \PP^1 ( \overline{\FF_q})$ with at most finite complements $M \setminus M_o$, $\PP^1 ( \overline{\FF_q}) \setminus L_o$, related by a finite surjective $\mathfrak{G}_m$-equivariant map $\xi : M_o \rightarrow L_o$.
  Moreover,  $\xi$  is required to have  a Galois closure, which amounts to the presence of  a commutative diagram of $\mathfrak{G}_s = {\rm Gal} ( \overline{\FF_q} / \FF_{q^s})$-equivariant maps
 \begin{equation} \label{IntroGaloisClosure}
 \begin{diagram}
 \node{N_o}  \arrow{e,t}{\xi _{H_1}}  \arrow{se,r}{\xi _H} \node{M_o}  \arrow{s,r}{\xi} \\
 \node{\mbox{  }}  \node{L_o}
  \end{diagram}.
 \end{equation}
for some $s \in m \NN$  with $\xi _H (z) = {\rm Orb} _H (z)$ and $\xi _{H_1} (z) = {\rm Orb} _{H_1} (z)$ for $\forall z \in N_o$.
 The third assumption from our sufficient condition for the Riemann Hypothesis Analogue  is formulated as an order of $\xi : M_o \rightarrow L_o$ and an $H$-order of $\xi _H : N_o \rightarrow L_o$.
    It  consists of  upper bounds on the number of the $\Phi _{q^s} ^r$-rational (i.e., $\Phi _{q^s} ^r$-fixed) points on $M_o$, respectively,  $N_o$, depending on the number of  $\Phi _{q^s} ^r$-rational points of $L_o$.

  While proving Theorem \ref{MainTheorem}, we observe that the Riemann Hypothesis Analogue for $M$ with respect to $\PP^1 ( \overline{\FF_q})$ as a locally finite module over the profinite completion  $\mathfrak{G} = {\rm Gal} ( \overline{\FF_q} / \FF_q) = \widehat{ \langle \Phi _q  \rangle}$ of $\langle \Phi _q \rangle \simeq (\ZZ, +)$ is equivalent to the same statement for $M$ as a locally finite module over  any  twist $\widehat{ \langle h \Phi _q^n \rangle}$ of the profinite action by a bijective self-map $h: M \rightarrow M$ of finite order in ${\rm Sym} (M)$, which centralizes the action of $\mathfrak{G}$  and any $n \in \NN$.
 We observe also  that the Riemann Hypothesis Analogue for $M$ with respect to $\PP^1 ( \overline{\FF_q})$ implies a specific functional equation for the polynomial $\zeta$-quoties $P_M(t) = \frac{\zeta _M(t)}{\zeta _{\PP^1 ( \overline{\FF_q})} (t) } \in \ZZ[t]$ of $M$.
   The work provides an explicit example of a locally finite $\mathfrak{G} = {\rm Gal} ( \overline{\FF_q} / \FF_q)$-module, subject to the Riemann Hypothesis Analogue with respect to $\PP^1 ( \overline{\FF_q})$, which is not isomorphic to a smooth irreducible projective variety $Y / \FF_q \subset \PP^n ( \overline{\FF_q})$ as a module over $\mathfrak{G}$.
Thus, our main Theorem \ref{MainTheorem} does not reduce to the Riemann Hypothesis Analogue for the smooth irreducible projective varieties, defined over $\FF_q$.

Here is a brief synopsis of the  article.
The next section 2 collects some trivial immediate properties of the locally finite $\mathfrak{G} = {\rm Gal} ( \overline{\FF_q} / \FF_q)$-modules $M$ and their $\mathfrak{G}$-equivariant maps.
Section 3   supplies several expressions of the $\zeta$-function $\zeta _M(t)$ of $M$ and shows that $\zeta _M(t)$ determines uniquely $M$, up to a $\mathfrak{G}$-equivariant isomorphism.
It studies the $\zeta$-quotient $P_M(t) = \frac{\zeta _M(t)}{\zeta _{\PP^1 ( \overline{\FF_q})} (t)} \in \ZZ [[t]]$ of $M$  and provides two necessary and sufficient conditions for $P_M(t) \in \ZZ[t]$ to be a polynomial.
Section 4 discusses the finite coverings of locally finite $\mathfrak{G}$-modules and their Galois closures.
It establishes that a locally finite $\mathfrak{G}$-module $M$ is faithful if and only if $M$ is an infinite set.
After introducing the notion of a finite Galois covering $\xi _H : N \rightarrow L$ of locally finite $\mathfrak{G}$-modules, we establish that a finite separable surjective morphism $\xi : X \rightarrow Y$ of smooth irreducible projective curves is an $H$-Galois covering  of locally finite
 $\mathfrak{G}_m = {\rm Gal} ( \overline{\FF_q} / \FF_{q^m})$-modules for some $m \in \NN$ exactly when the induced embedding
  $\xi ^* : \overline{\FF_q} (Y) \hookrightarrow \overline{\FF_q} (X)$ of the function fields is a Galois extension with
Galois group   ${\rm Gal} ( \overline{\FF_q} (X) / \xi ^* \overline{\FF_q} (Y)) = H$.
  That explains the etymology of the notion of a finite Galois covering of locally finite $\mathfrak{G}$-modules.
  The definition of a Galois closure of a finite covering of locally finite $\mathfrak{G}$-modules arises from  the case of a non-constant separable morphism
  $\xi : X \rightarrow \PP^1( \overline{\FF_q})$ of smooth irreducible projective curves, when the Galois closure $\overline{\FF_q} (N_o) \supseteq \overline{\FF_q}(\PP^1 (\overline{\FF_q}))$ of $\overline{\FF_q} (X) \supseteq  \overline{\FF_q}(\PP^1 (\overline{\FF_q}))$  is associated with  a smooth irreducible quasi-projective curve $N_o$, subject to a commutative diagram (\ref{IntroGaloisClosure}) for some $\mathfrak{G}_s$-submodules $M_o \subseteq X$, $L_o \subseteq \PP^1 ( \overline{\FF_q})$ with $\left| X \setminus M_o \right| < \infty$, $\left| \PP^1 ( \overline{\FF_q}) \setminus L_o \right| < \infty$ and
  $H = {\rm Gal} ( \overline{\FF_q} (N_o) / \overline{\FF_q} ( \PP^1 ( \overline{\FF_q})))$,
  $H_1 = {\rm Gal} ( \overline{\FF_q} (N_o) / \overline{\FF_q} (X) )$.
 The final, fifth section is devoted to the main result of the article.
  After reducing the Riemann Hypothesis Analogue with respect to $\PP^1 ( \overline{\FF_q})$ for a locally finite
  $\mathfrak{G} = {\rm Gal} ( \overline{\FF_q} / \FF_q)$-module $M$ to   lower and upper bounds on the number of rational points of $M$, we discuss the notion of an order of a finite covering $\xi : M_o \rightarrow L_o$ and the notion of an $H$-order of an $H$-Galois covering $\xi _H : N_o \rightarrow L_o$.
  Then we prove a sufficient condition for a locally finite $\mathfrak{G}$-module to satisfy the Riemann Hypothesis Analogue with respect to the projective line $\PP^1 ( \overline{\FF_q})$.
 An explicit example  establishes the existence of   locally finite $\mathfrak{G}$-modules, subject to the Riemann Hypothesis Analogue with respect to $\PP^1 ( \overline{\FF_q})$, which  are not isomorphic  (as  modules over  $\mathfrak{G}$) to   smooth irreducible projective varieties, defined over $\FF_q$.

%%%%%%%%%%%%%%%%%%%%%%%%%%%%%%%%%%%%%%%%%%%%%%%%%%%%%%%%%%%%%%%%%%%%%%%%%%%%%%%%%%%%%%%%%%%%%%%%%%%%%%%%%%%%%%%%%%%%%%%%%%%%%%%%%%%%%%%555
\section{ Preliminaries on locally finite $\widehat{\ZZ}$-modules  \\   and  their equivariant maps  }

The present section collects some trivial facts about the locally finite modules over $\mathfrak{G} = {\rm Gal} ( \overline{\FF_q} / \FF_q) \simeq \widehat{\ZZ}$ and their $\mathfrak{G}$-equivariant maps.
For an arbitrary $n \in \NN$, note that
$$
\mathfrak{G} \times \PP^n ( \overline{\FF_q}) \longrightarrow \PP^n ( \overline{\FF_q}),
$$
$$
( \Phi _q ^{l_s ({\rm mod} \, s)} ) _{s \in \NN} [ a_0 : \ldots : a_i : \ldots a_n ] =
[ a_0 ^{q^{l_s}} : \ldots  : a_n ^{q^{l_s}} ]  \ \ \mbox{  for  } \ \ a_0, \ldots, a_n \in \FF_{q^s}
$$
is a correctly defined action with finite orbits  by Remark 2.1.10 (i) and Lemma 2.1.9 from \cite{NX}.
 According to  Lemma 2.1.11 from \cite{NX},   $\deg {\rm Orb} _{\mathfrak{G}} (a) = m$ for some $a \in \PP^n ( \overline{\FF_q})$  exactly when
 $
 {\rm Orb} _{\mathfrak{G}} (a) = {\rm Orb} _{\langle \Phi _q \rangle} (a) = \{ \Phi _q ^{j} (a) \, \vert \, 0 \leq j \leq m-1 \}.
 $
This, in turn, is equivalent to     $\FF_q \left( \frac{a_0}{a_i}. \ldots \frac{a_n}{a_i} \right) = \FF_{q^m}$ for any $0 \leq i \leq n$ with $a_i \neq 0$.
 Moreover,    $m \in \NN$ is the minimal natural number with $\Phi _q ^m (a) = a$.
 For   $\forall m \in \NN$ there are at most finitely many $\mathfrak{G}$-orbits of degree $m$ on $\PP^n ( \overline{\FF_q})$, as far as an arbitrary standard affine open  subset $U_i := \{ a = [ a_0 : \ldots : a_n] \in \PP^n ( \overline{\FF_q}) \, \vert \, a_i \neq 0 \}$, $0 \leq i \leq n$ contains finitely many points with $\FF_q \left( \frac{a_0}{a_i}, \ldots , \frac{a_n}{a_i} \right) = \FF_{q^m}$.
Thus, $\PP^n ( \overline{\FF_q})$ is a locally finite $\mathfrak{G}$-module.

If $X = V(f_1, \ldots , f_l) \subset \PP^n ( \overline{\FF_q})$ is a smooth irreducible curve, cut by homogeneous  polynomials
 $f_1, \ldots , f_l \in \FF_q [ x_0, \ldots , x_n]$  of one and a same degree with coefficients from $\FF_q$, $X$ is said to be defined  over $\FF_q$ and denoted by
 $X / \FF_q \subset \PP^n ( \overline{\FF_q})$.
 The $\mathfrak{G}$-action on $\PP^n ( \overline{\FF_q})$ restricts to a locally finite  $\mathfrak{G}$-action on $X$, due to the $\mathfrak{G}$-invariance  of $f_1, \ldots , f_l$.

The following lemma verifies that   the aforementioned properties  of the $\mathfrak{G}$-action on
 $X / \FF_q \subset \PP^n ( \overline{\FF_q})$ are valid for an arbitrary locally finite action of $\widehat{\ZZ}$.

\begin{lemma}   \label{OrbStab}
Let $\mathfrak{G} = \widehat{\langle \varphi \rangle}$ be the profinite completion of an infinite cyclic group $\langle \varphi \rangle \simeq ( \ZZ, +)$, $M$ be a locally finite $\mathfrak{G}$-module, ${\rm Orb} _{\mathfrak{G}} (x) \subseteq M$ be a $\mathfrak{G}$-orbit on $M$ of degree
$m = \deg {\rm Orb} _{\mathfrak{G}} (x)$ and $\mathfrak{G}_m = \widehat{\langle \varphi ^m \rangle}$ be the profinite completion of
$\langle \varphi ^m \rangle \simeq ( \ZZ, +)$.
Then:

(i)  any $y \in {\rm Orb} _{\mathfrak{G}} (x)$ has stabilizer
$
{\rm Stab} _{\mathfrak{G}} (y) = {\rm Stab} _{\mathfrak{G}} (x) = \mathfrak{G}_m;
$

(ii) the  orbits
$
{\rm Orb} _{\mathfrak{G}} (x) = {\rm Orb} _{\langle \varphi \rangle} (x) = \{  x, \varphi (x), \ldots , \varphi ^{m-1} (x) \}
$
 coincide;

(iii) for any $r \in \NN$ with ${\rm GCD} (r,m) = d \in \NN$, the $\mathfrak{G}$-orbits
 $$
 {\rm Orb} _{\mathfrak{G}} (x) = \coprod\limits _{j=1} ^d {\rm Orb} _{\mathfrak{G}_r} ( \varphi ^{i_j} (x))
 $$
  of $x$  decomposes into a disjoint union of $d$  orbits of degree $m_1 = \frac{m}{d}$ with respect to the action of
   $\mathfrak{G}_r = \widehat{\langle \varphi ^r \rangle}$.
\end{lemma}

\begin{proof}

If $\mathfrak{G}':= {\rm Gal} ( \overline{\FF_q} / \FF_q) = \widehat{ \langle \Phi _q \rangle}$ is the absolute Galois group of the finite field $\FF_q$,
 then the group isomorphism
$$
f: \langle \varphi \rangle \longrightarrow \langle \Phi _q \rangle,   \quad
f( \varphi ^s) = \Phi _q ^s \ \ \mbox{ for } \ \ \forall s \in \NN
$$
extends uniquely to a group isomorphism
$$
f: \mathfrak{G} = \widehat{\langle \varphi \rangle} \longrightarrow \widehat{\langle \Phi _q \rangle} = \mathfrak{G}',  \quad
f( \varphi ^{l_s ( {\rm  mod}\, s)} ) _{s \in \NN} = ( \Phi _q ^{l_s ({\rm mod} \, s)} ) _{s \in \NN} \in
\prod\limits _{s \in \NN} ( \langle \Phi _q \rangle / \langle \Phi _q ^s \rangle).
$$
That is why it suffices to prove the lemma for $\mathfrak{G}'= \widehat{ \langle \Phi _q \rangle}$.

(i)  Note that $\FF_{q^m} \supset \FF_q$ is a finite Galois extension with Galois group
 ${\rm Gal} ( \FF_{q^m} / \FF_q) = \langle \Phi _q \rangle / \langle \Phi _q ^m \rangle \simeq ( \ZZ_m, +)$.
Bearing in mind that $\overline{\FF_q} = \overline{\FF_{q^m}} \supset \FF_{q^m} \supset \FF_q$ is a series of Galois extensions, one concludes that
$
{\rm Gal} ( \FF_{q^m} / \FF_q) = {\rm Gal} ( \overline{\FF_q} / \FF_q) / {\rm Gal} ( \overline{\FF_q} / \FF_{q^m}) = \mathfrak{G}'/ \mathfrak{G}'_m
$
for the profinite completion $\mathfrak{G}'_m = \widehat{ \langle \Phi _q ^m \rangle}$ of $\langle \Phi _q ^m \rangle \simeq ( \ZZ, +)$.
In particular, the closure $\mathfrak{G}'_m = \overline{\langle \Phi _q ^m \rangle}$ of $\langle \Phi _q ^m \rangle$ in $\mathfrak{G}'$ is a closed subgroup of $\mathfrak{G}'$ of index $[ \mathfrak{G}': \mathfrak{G}'_m] = m$ and $\mathfrak{G}_m = \widehat{ \langle \varphi ^m \rangle}$ is a closed subgroup of $\mathfrak{G} = \widehat{ \langle \varphi \rangle}$ of index $[ \mathfrak{G} : \mathfrak{G}_m] = m$.

We claim that $\mathfrak{G}_m$ is the only closed subgroup of $\mathfrak{G}$ of index $m$.
To this end, if $\cH$ is a closed subgroup of $\mathfrak{G}$ of $[ \mathfrak{G} : \cH] = m$ then $\mathfrak{G} / \cH$ is an abelian group of order $m$ and its element $\varphi \cH \in \mathfrak{G} / \cH$ has $( \varphi \cH) ^m = \varphi ^m \cH = \cH$.
 Thus, $\varphi ^m \in \cH$ implies the presence of a sequence of groups
  $\mathfrak{G} _m = \overline{\langle \varphi ^m \rangle} \subseteq \cH \subseteq \mathfrak{G}$.
  The index
$
[ \cH : \mathfrak{G}_m] = \frac{[ \mathfrak{G} : \mathfrak{G}_m]}{[ \mathfrak{G} : \cH]} = 1,
$
 so that  $\cH = \mathfrak{G}_m$.
In particular, for any $y \in {\rm Orb} _{\mathfrak{G}} (x)$ one has ${\rm Orb} _{\mathfrak{G}} (y) = {\rm Orb} _{\mathfrak{G}} (x)$ and  the closed subgroup ${\rm Stab} _{\mathfrak{G}} (y) \leq \mathfrak{G}$ of index  $[ \mathfrak{G} : {\rm Stab} _{\mathfrak{G}} (y)] = \deg {\rm Orb} _{\mathfrak{G}} (y) =m$ coincides with $\mathfrak{G}_m$.

(ii)   The inclusion $\langle \varphi \rangle \subset \widehat{ \langle \varphi \rangle} = \mathfrak{G}$ of groups implies the inclusion
 ${\rm Orb} _{\langle \varphi \rangle} (x) \subseteq {\rm Orb} _{\mathfrak{G}} (x)$ of the corresponding orbits.
 It suffices to show that $x, \varphi (x), \ldots , \varphi ^{m-1} (x)$ are pairwise different, in order to conclude that
  $\deg {\rm Orb} _{\langle \varphi \rangle} (x) \geq m = \deg {\rm Orb} _{\mathfrak{G}} (x)$,  whereas
   ${\rm Orb} _{\langle \varphi \rangle} (x) = {\rm Orb} _{\mathfrak{G}} (x)$.
 Indeed, if $\varphi ^{i} (x) = \varphi ^{j} (x)$ for some $0 \leq i < j \leq m-1$ then $x = \varphi ^{j-i} (x)$ implies
 $\varphi ^{j-i} \in  {\rm Stab}_{\mathfrak{G}} (x) \cap \langle \varphi \rangle = \widehat{\langle \varphi \rangle} \cap \langle \varphi \rangle =
  \langle \varphi ^m \rangle$ and $m$ divides $0 < j-i \leq m-1$.
 This is an absurd, justifying  the coincidence  ${\rm Orb} _{\langle \varphi \rangle} (x) = {\rm Orb} _{\mathfrak{G}} (x)$.

 (iii)  It suffices to prove that   $\forall y \in {\rm Orb} _{\mathfrak{G}} (x)$ has stabilizer ${\rm Stab} _{\mathfrak{G}_r} (y) = \mathfrak{G} _{rm_1} \leq \mathfrak{G} _r \leq \mathfrak{G}$, in order to apply (i) and to conclude that
 $$
 \deg {\rm Orb} _{\mathfrak{G}_r} (y) = [ \mathfrak{G}_r : {\rm Stab} _{\mathfrak{G}_r} (y)] =
 \frac{[\mathfrak{G} : {\rm Stab} _{\mathfrak{G}_r} (y)]}{[ \mathfrak{G} : \mathfrak{G}_r]} = \frac{rm_1}{r} = m_1.
 $$
 Bearing in mind that ${\rm Stab}_{\mathfrak{G}_r} (y) = {\rm Stab} _{\mathfrak{G}} (y) \cap \mathfrak{G}_r = \mathfrak{G}_m \cap \mathfrak{G}_r$ and
 ${\rm LCM} (m,r) = \frac{mr}{{\rm GCD} (m,r)} = r m_1 = m r_1$ for $r_1 = \frac{r}{d}$, we reduce the statement to $\mathfrak{G}_m \cap \mathfrak{G}_r = \mathfrak{G} _{{\rm LCM} (m,r)}$ for arbitrary $m, r \in \NN$.
 The quotient group $\mathfrak{G}_r / \mathfrak{G}_m \cap \mathfrak{G}_r \simeq \mathfrak{G}_r \mathfrak{G}_m / \mathfrak{G}_m$ is finite as a subgroup of $\mathfrak{G} / \mathfrak{G}_m \simeq (\ZZ_m, +)$.
 Therefore $\mathfrak{G}_m \cap \mathfrak{G}_r$ is a closed subgroup of $\mathfrak{G}$ of finite index
 $$
 [ \mathfrak{G} : \mathfrak{G} _m \cap \mathfrak{G}_r] = [ \mathfrak{G} : \mathfrak{G}_r]
 [ \mathfrak{G}_r : \mathfrak{G}_m \cap \mathfrak{G}_r] \leq rm
 $$
  in $\mathfrak{G}$ and $\mathfrak{G}_m \cap \mathfrak{G}_r = \mathfrak{G}_s$ for
 $s := [ \mathfrak{G} : \mathfrak{G}_m \cap \mathfrak{G}_r]$.
The inclusions $\mathfrak{G}_s \leq \mathfrak{G}_m \leq \mathfrak{G}$ imply that $s = [ \mathfrak{G} : \mathfrak{G}_s] = [ \mathfrak{G} : \mathfrak{G}_m]
[ \mathfrak{G}_m : \mathfrak{G}_s]  = m [ \mathfrak{G}_m : \mathfrak{G}_s]$ is a multiple of $m$.
For similar reasons, $s$ is a multiple of $r$ and ${\rm LCM} (m,r) \in \NN$ divides $s$.
On the other hand, $\mathfrak{G}_{{\rm LCM} (m,r)} = \widehat{ \langle \varphi ^{m r_1} \rangle} = \widehat{ \langle \varphi ^{r m_1} \rangle}$ is contained in $\mathfrak{G}_m = \widehat{ \langle \varphi ^m \rangle}$ and $\mathfrak{G}_r = \widehat{\langle \varphi ^r \rangle}$, whereas $\mathfrak{G} _{{\rm LCM} (m,r)} \leq \mathfrak{G}_m \cap \mathfrak{G}_r = \mathfrak{G}_s \leq \mathfrak{G}$.
As a result, $s$ divides ${\rm LCM} (m,r)$ and $s = {\rm LCM} (m,r)$.

\end{proof}

The next proposition provides a numerical description of a $\mathfrak{G}$-equivariant map $\xi : M \rightarrow L$, $\xi ( \gamma y) = \gamma \xi (y)$ for $\forall \gamma \in \mathfrak{G}$, $\forall y \in M$ on a preimage of a $\mathfrak{G}$-orbit, by the means of the inertia indices of $\xi$.
Note that the image $\xi (M)$ is $\mathfrak{G}$-invariant and for any complete set $\Sigma _{\mathfrak{G}} ( \xi (M)) \subseteq \xi (M)$ of $\mathfrak{G}$-orbit representatives, the $\mathfrak{G}$-orbit decomposition $\xi (M) = \coprod\limits_{x \in \Sigma _{\mathfrak{G}} ( \xi (M))} {\rm Orb} _{\mathfrak{G}} (x)$ pulls back to a disjoint $\mathfrak{G}$-module decomposition
\begin{equation}   \label{PullBackModuleDecomposition}
M = \coprod\limits _{x \in \Sigma _{\mathfrak{G}} ( \xi (M))} \xi ^{-1} {\rm Orb} _{\mathfrak{G}} (x).
\end{equation}
Thus, $\xi : M \rightarrow L$ is completely determined by its surjective $\mathfrak{G}$-equivariant restrictions
$$
\xi : \xi ^{-1} {\rm Orb} _{\mathfrak{G}} (x) \longrightarrow {\rm Orb} _{\mathfrak{G}} (x) \ \ \mbox{  for } \ \
 \forall x \in \Sigma _{\mathfrak{G}} ( \xi (M)).
$$

\begin{proposition-definition}    \label{InertiaIndices}
 For an arbitrary   $\mathfrak{G}$-equivariant map $\xi:  M \rightarrow L$    of locally finite modules over $\mathfrak{G} = \widehat{ \langle \varphi \rangle}$,  let us consider the maps
 $$
 \delta = \deg {\rm Orb} _{\mathfrak{G}} : L \longrightarrow \NN,  \ \     \delta (x) = \deg {\rm Orb} _{\mathfrak{G}} (x) \ \ \mbox{ for } \ \
 \forall x \in L \ \ \mbox{  and }
 $$
$$
e_{\xi} : M \longrightarrow \QQ ^{>0}, \ \
 e_{\xi} (y) = \frac{\deg {\rm Orb} _{\mathfrak{G}} (y)}{\deg {\rm Orb} _{\mathfrak{G}} ( \xi (y))} \ \ \mbox{  for } \ \ \forall y \in M.
$$
Then:

(i) ${\rm Stab} _{\mathfrak{G}} (y)  \leq  {\rm Stab} _{\mathfrak{G}} ( \xi (y)) \leq \mathfrak{G}$ for $\forall y \in M$ and
$e_{\xi} (y) = [ {\rm Stab} _{\mathfrak{G}} ( \xi (y)) : {\rm Stab} _{\mathfrak{G}} (y)] \in \NN$ takes natural values;

(ii)  for any $x \in \xi (M)$ there is a subset $S_x \subseteq \xi ^{-1} (x)$, such that
\begin{equation}    \label{OrbitDecompositionOfPulledBackOrbit}
\xi ^{-1} {\rm Orb} _{\mathfrak{G}} (x) = \coprod\limits _{y \in S_x} {\rm Orb} _{\mathfrak{G}} (y) \ \ \mbox{  with } \ \
 \deg {\rm Orb} _{\mathfrak{G}} (y) = \delta (x) e_{\xi} (y);
\end{equation}

(iii)  for any $x \in \xi (M)$ the fibre $\xi ^{-1} (x)$ is a $\mathfrak{G}_{\delta (x)}$-module with orbit decomposition
\begin{equation}    \label{OrbitDecompositionOfFibre}
\xi ^{-1} (x) = \coprod\limits _{y \in S_x} {\rm Orb} _{\mathfrak{G}_{\delta (x)}}  (y) \ \ \mbox{ of } \ \
 \deg {\rm Orb} _{\mathfrak{G}_{\delta (x)}}  (y) = e_{\xi} (y).
\end{equation}

The correspondence $e_{\xi} : M \rightarrow \NN$ will be referred to as the inertia map of $\xi : M \rightarrow L$.
The values $e_{\xi} (y)$, $y \in M$ of $e_{\xi}$ are called inertia indices of $\xi$.
\end{proposition-definition}

\begin{proof}

(i) For arbitrary $y \in M$ and $\gamma \in {\rm Stab} _{\mathfrak{G}} (y)$, note that $\gamma \xi (y) = \xi ( \gamma y) = \xi (y)$, in order to conclude that ${\rm Stab} _{\mathfrak{G}} (y) \leq {\rm Stab} _{\mathfrak{G}} ( \xi (y)) \leq \mathfrak{G}$.
By Lemma \ref{OrbStab}(i), if $\lambda (y) := \deg {\rm Orb} _{\mathfrak{G}} (y) $ then
${\rm Stab} _{\mathfrak{G}} (y) = \mathfrak{G}_{\lambda (y)} = \widehat{ \langle \varphi ^{\lambda (y)} \rangle}$,
${\rm Stab} _{\mathfrak{G}} ( \xi (y)) = \mathfrak{G} _{\delta (\xi (y))} = \widehat{ \langle \varphi ^{\delta ( \xi (y))} \rangle}$,
so that the inertia index
$$
e_{\xi} (y) := \frac{\deg {\rm Orb} _{\mathfrak{G}} (y)}{\deg {\rm Orb} _{\mathfrak{G}} ( \xi (y))} = \frac{\lambda (y)}{\delta ( \xi (y))} =
\frac{[ \mathfrak{G} : {\rm Stab} _{\mathfrak{G}} (y)]}{[ \mathfrak{G} : {\rm Stab} _{\mathfrak{G}} ( \xi (y))]} =
[ {\rm Stab} _{\mathfrak{G}} ( \xi (y)) : {\rm Stab} _{\mathfrak{G}} (y)]  \in \NN
$$
is a natural number.

(ii If $x \in \xi (M)$ then  an arbitrary  $\mathfrak{G}$-orbit on $\xi ^{-1} {\rm Orb} _{\mathfrak{G}} (x)$ is claimed to  intersect the fibre
 $\xi ^{-1} (x)$.
Indeed,   $\xi (z) = \varphi ^s (x)$ for some $z \in M$ and $0 \leq s \leq \delta (x) -1$  implies
 $\xi ( \varphi ^{\delta (x) -s} z) = \varphi ^{ \delta (x) -s} \xi (z)  = x$, whereas
  $y:= \varphi ^{ \delta (x) -s} (z) \in \xi ^{-1} (x)$  with ${\rm Orb} _{\mathfrak{G}} (z) = {\rm Orb} _{\mathfrak{G}} (y)$.
 That allows to choose a complete set $S_x \subseteq \xi ^{-1} (x)$ of $\mathfrak{G}$-orbit representatives on $\xi ^{-1} {\rm Orb} _{\mathfrak{G}} (x)$ and to obtain (\ref{OrbitDecompositionOfPulledBackOrbit}) with   $\deg {\rm Orb} _{\mathfrak{G}} (y) = \lambda (y) = \delta (x) e_{\xi (y)}$ by (i).

 (iii)  If $x \in \xi (M)$, $\deg {\rm Orb} _{\mathfrak{G}} (x) = \delta (x)$ and $y \in \xi ^{-1} (x)$ then
  $\xi ( \varphi ^{\delta (x)} (y)) = \varphi ^{\delta (x)} \xi (y) = \varphi ^{\delta (x)} (x) = x$, whereas
  $\varphi ^{\delta (x)} (y) \in \xi ^{-1} (x)$ and $\mathfrak{G} _{\delta (x)} = \widehat{ \langle \varphi ^{\delta (x)} \rangle}$ acts on $\xi ^{-1} (x)$.
  That justifies the inclusion $\cup _{y \in S_x} {\rm Orb} _{\mathfrak{G} _{\delta (x)}} (y) \subseteq \xi ^{-1} (x)$.
  For any $y, y'\in S_x$ the assumption $y'\in {\rm Orb} _{\mathfrak{G} _{\delta (x)}} (y) \subseteq {\rm Orb} _{\mathfrak{G}} (y)$ implies that $y'= y$, so that the union $\coprod\limits _{y \in S_x} {\rm Orb} _{\mathfrak{G}_{\delta (x)}} (y)$ is disjoint.
  By the very definition of $S_x$, any
   $z \in \xi ^{-1} (x) \subset \xi ^{-1} {\rm Orb} _{\mathfrak{G}} (x) = \coprod\limits _{y \in S_x} {\rm Orb} _{\mathfrak{G}} (y)$ is of the form
   $z = \varphi ^s (y)$ for some $y \in S_x$ and $0 \leq s \leq \delta (x) e_{\xi} (y)$.
 Due  to $x = \xi (z) = \xi ( \varphi ^s (y))  = \varphi ^s \xi (y) = \varphi ^s (x)$, there follows
   $\varphi ^s \in {\rm Stab} _{\mathfrak{G}} (x) \cap \langle \varphi \rangle = \widehat{\langle  \varphi ^{\delta (x)} \rangle} \cap \langle \varphi \rangle =
    \langle \varphi ^{\delta (x)} \rangle$, whereas  $s = \delta (x) r$ for some  $r \in \ZZ^{\geq 0}$.
    Thus, $z = \varphi ^{\delta (x)r} (y) \in {\rm Orb} _{\mathfrak{G}_{\delta (x)}} (y)$ and
    $\xi ^{-1} (x) \subseteq \coprod\limits _{y \in S_x} {\rm Orb} _{\mathfrak{G} _{\delta (x)}} (y)$.
   That justifies the $\mathfrak{G}_{\delta (x)}$-orbit decomposition (\ref{OrbitDecompositionOfFibre}).
By (ii) and the proof of Lemma \ref{OrbStab} (iii),  one has
${\rm Stab} _{\mathfrak{G} _{\delta (x)}} (y) = {\rm Stab} _{\mathfrak{G}} (y) \cap \mathfrak{G}_{\delta (x)} =
 \mathfrak{G} _{\delta (x) e_{\xi} (y)} \cap \mathfrak{G}_{\delta (x)} = \mathfrak{G}_{\delta (x) e_{\xi} (y)}$,
  as far as  ${\rm LCM} ( \delta (x) e_{\xi} (y), \delta (x)) = \delta (x) e_{\xi} (y)$.
 Now, Lemma \ref{OrbStab}(i) applies to provide
 $\deg {\rm Orb} _{\mathfrak{G} _{\delta (x)}} (y) = [ \mathfrak{G} _{\delta (x)} : {\rm Stab} _{\mathfrak{G}_{\delta (x)}} (y)] =
   \frac{[\mathfrak{G} : \mathfrak{G} _{\delta (x) e_{\xi} (y)}]}{[\mathfrak{G} : \mathfrak{G} _{\delta (x)}]} =   e_{\xi} (y)$.

\end{proof}

%%%%%%%%%%%%%%%%%%%%%%%%%%%%%%%%%%%%%%%%%%%%%%%%%%%%%%%%%%%%%%%%%%%%%%%%%%%%%%%%%%%%%%%%%%%%%%%%%%%%%%%%%%%%
\section{ Locally finite modules with a polynomial $\zeta$-quotient }

In order to provide two more expressions for the $\zeta$-function of a locally finite module $M$ over $\mathfrak{G} = \widehat{ \langle \varphi \rangle}$,
$\langle \varphi \rangle \simeq ( \ZZ, +)$, let us recall that on an arbitrary smooth irreducible curve $X / \FF_q \subseteq \PP^n ( \overline{\FF_q})$, defined over $\FF_q$, the fixed points
$$
X^{\Phi _q ^r} := \{ x \in X \, \vert \,  \Phi _q ^r (x) = x \} = X( \FF_{q^r})
$$
of an arbitrary power $\Phi _q  ^r$, $r \in \NN$ of the Frobenius automorphism $\Phi _q$ coincide with the $\FF_{q^r}$-rational ones.
That is why, for an arbitrary bijective self-map $\varphi : M \rightarrow M$, which is of infinite order in the symmetric group ${\rm Sym} (M)$ of $M$, the fixed points
$$
M^{\varphi ^r} := \{ x \in M \, \vert \, \varphi ^r (x) = x \}
$$
of   $\varphi ^r$ with  $r \in \NN$  are called $\varphi ^r$-rational.
Note that if $\deg {\rm Orb} _{\mathfrak{G}} (x) = m$ then $x \in M^{\varphi ^r}$ if and only if $\varphi ^r \in {\rm Stab} _{\mathfrak{G}} (x) = \mathfrak{G} _m = \widehat{ \langle \varphi ^m \rangle}$ and this holds exactly when $m$ divides $r$.
Since any $r \in \NN$ has finitely many natural divisors $m$ and for any $m \in \NN$ there are at most finitely many $\mathfrak{G}$-orbits on $M$ of degree $m$, the sets $M^{\varphi ^r}$ are finite.

Let us consider the free abelian group $({\rm Div} (M), +)$, generated by the $\mathfrak{G}$-orbits $\nu \in {\rm Orb} _{\mathfrak{G}} (M)$.
Its elements are called divisors on $M$ and are of the form $D = a_1 \nu _1 + \ldots a_s \nu _s$ for some $\nu _j \in {\rm Orb} _{\mathfrak{G}} (M)$,
 $a_j \in \ZZ$.
 The terminology arises from the case of a smooth irreducible curve $X / \FF_q \subseteq \PP^n ( \overline{\FF_q})$, in which the
  $\mathfrak{G} = {\rm Gal} ( \overline{\FF_q} / \FF_q)$-orbits $\nu$  are in a bijective correspondence with the places $\widetilde{\nu}$ of the function field $\FF_q (X)$ of $X$ over $\FF_q$.
  If $R_{\nu}$ is the discrete valuation ring, associated with the place $\widetilde{\nu}$   then the residue field $R_{\nu} / \mathfrak{M}_{\nu}$ of $R_{\nu}$ is of degree $[ R_{\nu} / \mathfrak{M} _{\nu} : \FF_q] = \deg \nu = \left| \nu \right|$.

  Note that the degree of a $\mathfrak{G}$-orbit extends to a group homomorphism
  $$
  \deg : ({\rm Div} (M), +) \longrightarrow (\ZZ, +),
  $$
  $$
  \deg \left(  \sum\limits _{\nu \in {\rm Orb} _{\mathfrak{G}} (M)} a_{\nu} \nu \right) = \sum\limits _{\nu \in {\rm Orb} _{\mathfrak{G}} (M)} a_{\nu} \deg \nu.
  $$
  A divisor $D = a_1 \nu _1 + \ldots + a_s \nu _s \geq 0$ is effective if all of its non-zero coefficients are positive.
  Let ${\rm Div} _{\geq 0} (M)$ be the set of the effective divisors on $M$.
Note that  the effective divisors $D = a_1 \nu _1 + \ldots + a_s \nu _s \geq 0$ of degree
$\deg D = a_1 \deg \nu _1 + \ldots + a_s \deg \nu _s = m \in \ZZ^{\geq 0}$ have bounded coefficients $1 \leq a_j \leq m$ and bounded degrees $\deg \nu _j \leq m$ of the $\mathfrak{G}$-orbits from the support of $D$.
Bearing in mind that $M$ has at most finitely many $\mathfrak{G}$-orbits $\nu _j$ of degree $\deg \nu _j \leq m$, one concludes that there are at most finitely many effective divisors on $M$ of degree $m \in \ZZ ^{\geq 0}$ and denotes their number by $\cA _m (M)$.

The following lemma generalizes two of the well known expressions of the local Weil $\zeta$-function $\zeta _X(t)$ of a curve
 $X / \FF_q \subset \PP^n ( \overline{\FF_q})$ to the $\zeta$-function of  any  locally finite $\mathfrak{G} = \widehat{\langle \varphi \rangle}$-module $M$.
 The proofs are similar to the ones for   $X / \FF_q \subseteq \PP^n ( \overline{\FF_q})$, given in \cite{NX} or in \cite{St}.

\begin{proposition}    \label{DzetaFunctionExpressions}
Let $\mathfrak{G} = \widehat{ \langle \varphi \rangle}$ be the profinite completion of an infinite cyclic group $\langle \varphi \rangle \simeq (\ZZ, +)$ and $M$ be a locally finite $\mathfrak{G}$-module.
Then the $\zeta$-function of $M$ equals
$$
\zeta _M(t) = \exp \left( \sum\limits _{r=1} ^{\infty} \left| M ^{\varphi ^r} \right| \frac{t^r}{r} \right) = \sum\limits _{m=0} ^{\infty} \cA _m (M),
$$
where $\left| M ^{\varphi ^r} \right|$ is the number of $\varphi ^r$-rational points on $M$ and $\cA _m (M)$ is the number of the effective divisors on $M$ of degree $m \in \ZZ ^{\geq 0}$.
\end{proposition}

\begin{proof}

If $B_k (M)$ is the number of $\mathfrak{G}$-orbits on $M$ of degree $k$ then
$$
\zeta _M (t) := \prod\limits _{\nu \in {\rm Orb} _{\mathfrak{G}} (M)} \left( \frac{1}{1 - t ^{\deg \nu}} \right) =
\prod\limits _{k=1} ^{\infty} \left( \frac{1}{1 - t^k} \right) ^{B_k (M)} .
$$
Therefore
\begin{align*}
\log \zeta _M(t) = - \sum\limits _{k=1} ^{\infty} B_k (M) \log (1 - t^k) =
 \sum\limits _{k=1} ^{\infty} B_k (M) \left( \sum\limits _{n=1} ^{\infty} \frac{t^{kn}}{n} \right) =
  \sum\limits _{r=1} ^{\infty} \left( \sum\limits _{k / r} k B_k (M) \right) \frac{t^r}{r},
\end{align*}
according to
\begin{equation}   \label{LogarithmExpansion}
\log (1 - z) = - \sum\limits _{r=1} ^{\infty} \frac{z^r}{r} \in \QQ [[ z ]],
\end{equation}

If $M^{\varphi ^r} = \coprod\limits _{\deg {\rm Orb} _{\mathfrak{G}} (x) /r} {\rm Orb} _{\mathfrak{G}} (x)$ is the decomposition of $M^{\varphi ^r}$ into a disjoint union of $\mathfrak{G}$-orbits then the number of $\varphi ^r$-rational points on $M$ is
\begin{equation}   \label{RationalPointsAndOrbits}
\left| M^{\varphi ^r} \right| = \sum\limits _{k /r} k B_k (M),
\end{equation}
whereas   $\log \zeta _M(t) = \sum\limits _{r=1} ^{\infty} \left| M^{\varphi ^r} \right| \frac{t^r}{r}.$

On the other hand, note that
$$
\zeta _M(t) = \prod\limits _{\nu \in {\rm Orb} _{\mathfrak{G}} (M)} \left( \sum\limits _{n=0} ^{\infty} t^{\deg (n \nu)} \right) =
\sum _{D \in {\rm Div} _{\geq 0} (M)} t^{\deg D} = \sum\limits _{m=0} ^{\infty} \cA _m (M) t^m.
$$

\end{proof}

\begin{corollary}   \label{DzetaFunctionDeterminesModuleStructure}
Locally finite $\mathfrak{G} = \widehat{ \langle \varphi \rangle}$-modules $M$, $L$ admit a bijective $\mathfrak{G}$-equivariant map $\xi : M \rightarrow L$ if and only if their $\zeta$-functions $\zeta _M(t) = \zeta _L(t)$ coincide.
\end{corollary}

\begin{proof}

If $\xi : M \rightarrow L$ is a bijective $\mathfrak{G}$-equivariant map then for $\forall x \in L$ with $\deg {\rm Orb} _{\mathfrak{G}} (x) = \delta (x)$, then (\ref{OrbitDecompositionOfFibre}) from  Proposition-Definition \ref{InertiaIndices} (iii) provides a decomposition $\xi ^{-1} (x) = \coprod\limits _{y \in S_x} {\rm Orb} _{\mathfrak{G}_{\delta (x)}} (y)$ of the fibre $\xi ^{-1} (x)$ in a disjoint union of $\mathfrak{G}_{\delta (x)}$-orbits of $\deg {\rm Orb} _{\mathfrak{G}_{\delta (x)}} (y) = e_{\xi} (y)$.
Therefore $\left| S_x \right| =1$ for $\forall x \in L$, $e_{\xi} (y) =1$ for $\forall y \in M$ and
$\xi ^{-1} {\rm Orb} _{\mathfrak{G}} (x) = {\rm Orb} _{\mathfrak{G}} \xi ^{-1} (x)$ is of degree $\delta (x)$
by (\ref{OrbitDecompositionOfPulledBackOrbit}) from  Proposition-Definition \ref{InertiaIndices} (ii).
As a result, (\ref{PullBackModuleDecomposition}) takes the form $M = \coprod\limits _{x \in \Sigma _{\mathfrak{G}} (L)} {\rm Orb} _{\mathfrak{G}} \xi ^{-1} (x)$ for any complete set $\Sigma _{\mathfrak{G}} (L)$ of $\mathfrak{G}$-orbit representatives on $L$ and the $\zeta$-functions
$
\zeta _M(t) = \prod\limits _{x \in \Sigma _{\mathfrak{G}} (L)} \left( \frac{1}{1 - t^{\delta (x)}} \right) = \zeta _L(t)
$
coincide.

Conversely, assume that the locally finite $\mathfrak{G}$-modules $M$ and $L$ have one and a same $\zeta$-function $\zeta _M(t) = \zeta _L(t)$.
Then by Proposition \ref{DzetaFunctionExpressions}, there follows the equality
$$
\sum\limits _{r=1} ^{\infty} \left| M ^{\varphi ^r} \right| \frac{t^r}{r} =
 \log \zeta _M(t) =
 \log \zeta _L(t) =
  \sum\limits _{r=1} ^{\infty} \left| L^{\varphi ^r} \right| \frac{t^r}{r} \in \QQ [[ t ]]
$$
of formal power series of $t$ and the equalities
$$
\sum\limits _{d / r} d B_d (M) = \left| M ^{\varphi ^r} \right| = \left| L ^{\varphi ^r} \right| = \sum\limits _{d /r} d B_d (L)
$$
of their coefficients for $\forall r \in \NN$.
By an induction on the number of the prime divisors of $r$, counted with their multiplicities, one derives the equalities $B_d (M) = B_d (L)$ of the numbers of the $\mathfrak{G}$-orbits of degree $d$ for $\forall d \in \NN$.

For an arbitrary $k \in \NN$, note that   $M^{( \leq k)} := \{ x \in M \, \vert \, \deg {\rm Orb} _{\mathfrak{G}} (x) = \delta (x) \leq k \}$   is a $\mathfrak{G}$-submodule of $M$ and $M = \cup _{k \in \NN} M^{( \leq k)}$.
By an induction on $k$, we construct a bijective $\mathfrak{G}$-equivariant map $\xi : M ^{( \leq k)} \rightarrow L^{( \leq k)}$, in order to show the existence of a $\mathfrak{G}$-equivariant bijection $\xi : M \rightarrow L$.
More precisely, if $M^{( \leq 1)} = M^{\varphi} =  \left \{ y_1 ^{(1)}, \ldots , y_{B_1(L)}^{(1)} \right \}$ and
$L^{( \leq 1)} = L^{\varphi} =  \left \{ x_1 ^{(1)}, \ldots , x_{B_1 (L)} ^{(1)} \right  \}$ then we  set  $\xi  ( y_i ^{(1)} ) = x_i ^{(1)}$
for $\forall 1 \leq i \leq B_1 (L)$ and obtain a bijective $\mathfrak{G}$-equivariant map $\xi : M^{( \leq 1)} \rightarrow L^{( \leq 1)}$.
If $\xi : M^{( \leq k-1)}  \rightarrow L^{( \leq k-1)}$ is constructed by the inductional hypothesis and
$\{ x_i ^{(k)} \in L \, \vert \, 1 \leq i \leq B_k (L) \}$, respectively,  $\{ y_i ^{(k)} \in M \, \vert \, 1 \leq i \leq B_k (L) \}$ are complete sets of representatives of the $\mathfrak{G}$-orbits of degree $k$ on $L$, respectively, on $M$, then $\xi ( \varphi ^s ( y_i ^{(k)} )) = \varphi ^s ( x_i ^{(k)})$ for $\forall 0 \leq s \leq k-1$, $\forall 1 \leq i \leq B_k (L)$ provides a $\mathfrak{G}$-equivariant  bijective extension $\xi : M^{( \leq k)} \rightarrow L^{( \leq k)}$  of  $\xi : M^{( \leq k-1)} \rightarrow L^{( \leq k-1)}$.
Since  $M^{( \leq k)}$ exhaust $M$ and $L^{( \leq k)}$ exhaust $L$, that establishes the existence of a $\mathfrak{G}$-equivariant bijective map
 $\xi : M \rightarrow L$.

\end{proof}

Form now on, we identify the locally finite modules $M$ and $L$, which admit a  $\mathfrak{G}$-equivariant bijective map $\xi : M \rightarrow L$.

\begin{lemma}    \label{DzetaQuotient}
If $M$ is a locally finite $\mathfrak{G} = {\rm Gal} ( \overline{\FF_q} / \FF_q)$-module with $\zeta$-function $\zeta _M(t) \in \ZZ [[ t ]]$ then the quotient
$$
P_M(t) = \frac{\zeta _M(t)}{\zeta _{\PP^1 ( \overline{\FF_q})} (t)} = \sum\limits _{m=0} ^{\infty}  a  _m (M) t^m \in \ZZ [[ t ]] ^*
$$
is a formal power series with integral coefficients $a_m \in \ZZ$, which is invertible in $\ZZ [[ t ]]$.
Its coefficients $ a _m (M) \in \ZZ$ satisfy the equality
$$
\cA _m (M) = \sum\limits _{i=0} ^m  a  _i (M) \left| \PP ^{m-i} ( \FF_q) \right|
$$
and can be interpreted as ''multiplicities'' of the projective spaces  $\PP^{m-i} ( \FF_q)$, ''exhausting'' the effective divisors on $M$ of degree $m$.
\end{lemma}

\begin{proof}

If $P_M(t) = \sum\limits _{m=0} ^{\infty} a  _m (M) t^m \in \CC [[ t ]]$ is a formal power series with complex coefficients $ a  _m (M) \in \CC$ then the comparison of the coefficients of
$$
\sum\limits _{m=0} ^{\infty}  a  _m (M) t^m =
 P_M(t) =
  \zeta _M(t) (1-t)(1-qt) =
   \left( \sum\limits _{m=0} ^{\infty} \cA _m (M) t^m \right) [1 - (q+1)t + qt^2]
$$
yields
\begin{equation}   \label{KappaFormula}
 a  _m (M) = \cA _m (M) - (q+1) \cA _{m-1} (M) + q \cA _{m-2} (M) \in \ZZ \ \ \mbox{  for } \ \ \forall m \in \ZZ ^{\geq 0},
\end{equation}
as far as $\cA _m (M) \in \ZZ ^{\geq 0}$ for $\forall m \in \ZZ ^{\geq 0}$ and $\cA _{-1} (M) = \cA _{-2} (M) =0$.
In particular, $ a _0 (M) = \cA _0 (M) = \zeta _M(0) =1$ and $P_M(t) = 1 + \sum\limits _{m=1} ^{\infty}  a  _m (M) t^m \in \ZZ [[ t ]]^*$ is invertible by a formal power series $P_M^{-1} (t) = 1 + \sum\limits _{m=1} ^{\infty}  b_m (M) t^m \in \ZZ [[ t ]]$ with integral coefficients.
(The existence of $ b_m (M) \in \ZZ$ with $[1 + \sum\limits _{m=1} ^{\infty}  a  _m (M) t^m ][1 + \sum\limits _{m=1} ^{\infty}  b _m (M) t^m] =1$ follows from
$
 b  _m (M) + \sum\limits _{i=1} ^{m-1}  b  _i (M)  a  _{m-i} (M) +  a  _m (M) =0
$
by an induction on $m \in \NN$.)

The comparison of the coefficients of
$$
\sum\limits _{m=0} ^{\infty} \cA _m (M) t^m =
 \zeta _M(t) =
  P_M(t) \zeta _{\PP^1 ( \overline{\FF_q})} (t) =
\left( \sum\limits _{m=0} ^{\infty}  a   _m (M) t^m \right) \left( \sum\limits _{s=0} ^{\infty} t^s \right)
 \left( \sum\limits _{r=0} ^{\infty} q^r t^r \right)
$$
provides
\begin{equation}  \label{AmDecompositionFormula}
\cA _m (M) = \sum\limits _{i=0} ^m  a  _i (M) \left( \sum\limits _{j=0} ^{m-i} q^{j} \right) =
\sum\limits _{i=0} ^m  a _i (M) \left( \frac{q^{m-i+1} -1}{q-1} \right) =
\sum\limits _{i=0} ^m  a  _i (M) \left| \PP ^{m-i} ( \FF_q) \right|.
\end{equation}

\end{proof}

  According to the Riemann-Roch Theorem  for a divisor $D$ of degree $\deg D = n \geq 2g-1$  on  a smooth irreducible curve $X / \FF_q \subseteq \PP^n ( \overline{\FF_q})$ of genus $g \geq 0$, the linear equivalence class of $D$ is of dimension $n-g+1$ over $\FF_q$.
For any $n \in \ZZ ^{\geq 0}$ there exist one and a same number $h$ of linear equivalence classes of divisors on $X$ of degree $n$.
The natural number $h = P_X(1)$ equals the value of the $\zeta$-polynomial $P_X(t) = \frac{\zeta_X(t)}{\zeta _{\PP^1 ( \overline{\FF_q})} (t)} = \sum\limits _{j=0} ^{2g} a_j t^{j} \in \ZZ [t]$ of $X$ at $1$ and is called the class number of $X$.
Thus, for any natural number $n \geq 2g-1$ there are
$$
\cA _n (X)  = P_X(1) \left| \PP^{n-g} ( \FF_q) \right| = P_X(1) \left( \frac{q^{n-g+1}-1}{q-1} \right)
$$
effective divisors of $X$ of degree $n$.
Note that the $\zeta$-function $\zeta _X (t) = \frac{P_X(t)}{(1-t)(1-qt)}$ has residuums
 ${\rm Res} _{\frac{1}{q}} ( \zeta _X(t)) = \frac{P_X \left( \frac{1}{q} \right)}{1-q}$,
 ${\rm Res} _1 ( \zeta _X(t)) = \frac{P_X(1)}{q-1}$ at its simple poles $\frac{1}{q}$, respectively, $1$.
The $\zeta$-polynomial $P_X(t)$ satisfies the functional equation
$
P_X(t) = P_X \left( \frac{1}{qt} \right) q^g t^{2g},
$
according to Theorem 4.1.13 from \cite{NX} or to Theorem V.1.15 (b) from \cite{St}.
In particular, $P_X \left( \frac{1}{q} \right) = q^{-g} P_X(1)$ and
$$
\cA _n (X) = - q^{n+1} {\rm Res} _{\frac{1}{q}} ( \zeta _X(t)) - {\rm Res} _1 ( \zeta _X (t)) \ \ \mbox{  for } \ \ \forall n \geq 2g-1.
$$

\begin{definition}   \label{GenericRiemannRochConsitions}
A locally finite module $M$ over $\mathfrak{G} = {\rm Gal} ( \overline{\FF_q} / \FF_q)$ satisfies the Generic  Riemann-Roch Conditions if $M$ has
$$
\cA_n (M) = - q^{n+1} {\rm Res} _{\frac{1}{q}} ( \zeta _M(t)) - {\rm Res} _1 ( \zeta _M(t))
$$
effective divisors of degree $n$ for all the sufficiently large natural numbers $n \geq n_o$.
\end{definition}

One can compare the  Generalized Riemann-Roch Conditions  with the Polarized Rie\-mann-Roch Conditions from \cite{KM2}, which are shown to be equivalent to the Mac Williams identities for additive codes and, in particular, for linear codes over finite fields.

Here is a characterization of the locally finite $\mathfrak{G}$-modules $M$ with a polynomial $\zeta$-quotient
 $P_M(t) = \frac{\zeta _M(t)}{\zeta _{\PP^1 ( \overline{\FF_q}) } (t)} \in \ZZ [t]$.

 \begin{proposition}   \label{NSCPolynomialDzetaQuotient}
 The following conditions are equivalent for the $\zeta$-function $\zeta _M(t)$ of a locally finite module $M$ over
 $\mathfrak{G} = {\rm Gal} ( \overline{\FF_q} / \FF_q)$:

 (i) $P_M(t) := \frac{\zeta _M(t)}{\zeta _{\PP^1 ( \overline{\FF_q})} (t)} \in \ZZ [t]$ is a polynomial of $\deg P_M(t) = d \leq \delta \in \NN$;

 (ii)  $M$ satisfies the Generic Riemann-Roch Conditions
 \begin{equation}   \label{GRRCForM}
 \cA _n (M) = - q^{n+1} {\rm Res} _{\frac{1}{q}} ( \zeta _M(t))  - {\rm Res} _1 ( \zeta _M(t)) = \frac{q^{n+1} P_M \left( \frac{1}{q} \right) - P_M (t)}{q-1};
 \end{equation}

 \begin{equation}   \label{PowerSumFormula}
 {\rm (iii)} \ \ \left| \PP^1 ( \overline{\FF_q}) ^{\Phi _q ^r} \right| - \left| M^{\Phi _q ^r} \right| = \sum\limits _{j=1} ^d \omega  _j ^r \ \ \mbox{  for }
 \ \ \forall r \in \NN
 \end{equation}
 and some $\omega _j \in \CC ^*$, which turn to be the reciprocals of the roots of $P_M(t) = \prod\limits _{j=1} ^d (1  - \omega _jt)$.
 \end{proposition}

 \begin{proof}

$(i) \Rightarrow (ii)$ If $P_M(t) = \frac{\zeta _M(t)}{\zeta _{\PP^1 ( \overline{\FF_q})} (t)} = \sum\limits _{j=0} ^d a_j t^{j} \in \ZZ [t]$ is a polynomial of $\deg P_M(t) = d \leq \delta \in \NN$ then (\ref{AmDecompositionFormula}) reduces to
$$
\cA _m (M) = \sum\limits _{i=0} ^d a_i \left( \frac{q^{m-i+1}-1}{q-1} \right) = \frac{q^{m+1} P_M \left( \frac{1}{q} \right) - P_M(1)}{q-1} \ \ \mbox{  for } \ \ \forall m \geq \delta.
$$
Moreover, (\ref{AmDecompositionFormula}) implies that
$$
\cA_{\delta -1} (M) = \frac{q^{\delta} \left[ P_M \left( \frac{1}{q} \right) - \frac{a_{\delta}}{q^{\delta}} \right] - \left[ P_M(1) - a_{\delta} \right]}{q-1} = \frac{q^{\delta} P_M \left( \frac{1}{q} \right) - P_M(1)}{q-1}.
$$
Note that the residuums of $\zeta _M(t) = \frac{P_M(t)}{(1-t)(1-qt)}$ at its simple poles are
${\rm Res} _{\frac{1}{q}} \left( \frac{P_M(t)}{(1-t)(1-qt)} \right) = \frac{P_M \left( \frac{1}{q} \right)}{1-q}$, respectively,
${\rm Res} _1 \left( \frac{P_M(t)}{(1-t)(1-qt)}\right) = \frac{P_M(1)}{q-1}$, in order to derive (\ref{GRRCForM}).

$(ii) \Rightarrow (i)$ Plugging (\ref{GRRCForM}) in (\ref{KappaFormula}), one obtains $ a _m (M) =0$ for $\forall m \geq \delta +1$.
Therefore $P_M(t) = \sum\limits _{i=0} ^{\delta}  a _i (M) t^{i} \in \ZZ[t]$ is a polynomial of degree $\deg P_M(t) \leq \delta$.

$(i) \Rightarrow (iii)$  If $P_M(t) = \frac{\zeta _M(t)}{\zeta _{\PP^1 ( \overline{\FF_q})} (t)} \in \ZZ[t]$ is a polynomial of degree $\deg P_M (t) = d \leq \delta$, then $P_M(0) = \frac{\zeta _M(0)}{\zeta _{\PP^1 ( \overline{\FF_q})} (0)} = 1$ allows to express $P_M(t) = \prod\limits _{j=1} ^d (1 - \omega _jt)$ by some complex numbers $\omega _j \in \CC^*$.
According to Proposition \ref{DzetaFunctionExpressions},
\begin{equation}   \label{DzetaFunctionsFormula}
\zeta _M(t) = \exp \left( \sum\limits _{r=1} ^{\infty} \left| M ^{\Phi _q ^r} \right| \frac{t^r}{r} \right) \ \ \mbox{  and } \ \
\zeta _{\PP^1 ( \overline{\FF_q})} (t) = \exp \left( \sum\limits _{r=1} ^{\infty} \left| \PP^1 ( \overline{\FF_q}) ^{\Phi _q ^r} \right| \frac{t^r}{r} \right),
\end{equation}
whereas
$$
\sum\limits _{j=1} ^d \log (1 - \omega _j t) = \log P_M(t) = \log \zeta _M(t) - \log \zeta _{\PP^1 (\overline{\FF_q})} (t) =
\sum\limits _{r=1} ^{\infty} \left( \left| M^{\Phi _q ^r} \right| - \left| \PP^1 ( \overline{\FF_q}) ^{\Phi _q^r} \right| \right) \frac{t^r}{r}.
$$
Making use of (\ref{LogarithmExpansion}), one obtains
$
- \sum\limits _{r=1} ^{\infty} \left( \sum\limits _{j=1} ^d \omega  _j ^r \right) \frac{t^r}{r} =
\sum\limits _{r=1} ^{\infty} \left( \left| M ^{\Phi _q^r} \right| - \left| \PP^1 ( \overline{\FF_q}) ^{\Phi _q^r} \right| \right) \frac{t^r}{r}.
$
The comparison of the coefficients of $\frac{t^r}{r}$ for $\forall r \in \NN$ provides (\ref{PowerSumFormula}).

$(iii) \Rightarrow (i)$ Multiplying (\ref{PowerSumFormula}) by $\frac{t^r}{r}$, summing for $\forall r \in \NN$ and making use of  (\ref{DzetaFunctionExpressions}), (\ref{LogarithmExpansion}), one obtains
$
\log \zeta _{\PP^1 ( \overline{\FF_q})} (t) - \log \zeta _M(t) = - \sum\limits _{j=1} ^d \log (1 - \omega _jt).
$
The change of the sign, followed by an exponentiation reveals that $P_M(t) = \frac{\zeta _M(t)}{\zeta _{\PP^1 ( \overline{\FF_q})} (t)} = \prod\limits _{j=1} ^d (1 - \omega _jt) \in \ZZ[t]$.

 \end{proof}

\begin{corollary}   \label{FiniteParametrizationLFMWithPDzQ}
Let $M$ and $L$ be locally finite $\mathfrak{G} = {\rm Gal} ( \overline{\FF_q} / \FF_q)$-modules with polynomial $\zeta$-quotients
$P_M(t) = \frac{\zeta _M(t)}{\zeta _{\PP^1 ( \overline{\FF_q})} (t)}$, $P_L(t) = \frac{\zeta _L(t)}{\zeta _{\PP^1 ( \overline{\FF_q})}(t)} \in \ZZ[t]$ of degree $\deg P_M(t) \leq \delta$, $\deg P_L(t) \leq \delta$.
Then $M$ and $L$ admit a $\mathfrak{G}$-equivariant isomorphism $\xi : M \rightarrow L$ if and only if they have one and a same number $B_k (M) = B_k(L)$ of $\mathfrak{G}$-orbits of degree $k$ for all $1 \leq k \leq \delta$.
\end{corollary}

\begin{proof}

According to Corollary \ref{DzetaFunctionDeterminesModuleStructure}, it suffices to show the coincidence $\zeta _M(t) = \zeta _N(t)$ of the corresponding $\zeta$-functions.
Making use of Proposition \ref{DzetaFunctionExpressions},  one  reduces the problem to
$$
\sum\limits _{r=1} ^{\infty} \left| M ^{\Phi _q ^r} \right| \frac{t^r}{r} =
\log \zeta _M(t) = \zeta _L(t) =
\sum\limits _{r=1} ^{\infty} \left| L ^{\Phi _q ^r}  \right| \frac{t^r}{r}
$$
and justifies that $M$ and $L$ have one and a same number $\left| M ^{\Phi _q ^r} \right| = \left| L ^{\Phi _q ^r} \right|$ of $\Phi _q ^r$-rational points for $\forall r \in \NN$.
According to (\ref{RationalPointsAndOrbits}), for $\forall r \in \NN$, $r \leq \delta$ one has
$$
\left| M ^{\Phi _q ^r} \right| = \sum\limits _{k /r} k B_k (M) = \sum\limits _{k/r} k B_k (L) = \left| L ^{\Phi _q ^r} \right|.
$$
If $P_M(t) = \prod\limits _{j=1} ^{\delta} (1 - \omega  _j t)$, respectively, $P_L(t) = \prod\limits _{j=1} ^{\delta} (1 - \rho _jt)$ for some
$\omega _j, \rho _j \in \CC$ (which are not supposed to be non-zero) then (\ref{PowerSumFormula}) from Proposition \ref{NSCPolynomialDzetaQuotient} implies that
$$
S_r := \sum\limits _{j=1} ^{\delta}  \omega _j ^r = \left| \PP^1 ( \overline{\FF_q})  ^{\Phi _q ^r} \right| - \left| M ^{\Phi _q^r} \right| =
\left| \PP^1 ( \overline{\FF_q})  ^{\Phi _q ^r} \right| - \left| L ^{\Phi _q ^r} \right| = \sum\limits _{j=1} ^{\delta} \omega _j ^r = : S'_r \ \
\mbox{  for } \ \ \forall 1 \leq r \leq \delta.
$$
The power sums $S_r$ for $1 \leq r \leq \delta$ determine uniquely $S_r$ for $r \geq \delta$ by Newton formulae.
The same holds for $S'_r$ and  implies $S_r = S'_r$ for $\forall r \in \NN$.
Thus,
$$
\left| M ^{\Phi _q ^r} \right| = \left| \PP^1 ( \overline{\FF_q}) ^{\Phi _q ^r} \right| - S_r =
 \left| \PP^1 ( \overline{\FF_q}) ^{\Phi _q ^r} \right| - S'_r = \left| L ^{\Phi _q ^r} \right| \ \ \mbox{  for } \ \ \forall r \in \NN,
 $$
  whereas $\zeta _M(t) = \zeta _L(t)$.

\end{proof}

\begin{proposition}   \label{Hierarchy}
Let $M$ be a locally finite module over the profinite completion $\mathfrak{G} = \widehat{ \langle \varphi \rangle}$ of
 $\langle \varphi \rangle \simeq ( \ZZ, +)$ and $M_r$ be the locally finite $\mathfrak{G}_r = \widehat{ \langle \varphi ^r \rangle}$-module, supported by $M$ for some $r \in \NN$.
 Then the $\zeta$-functions of $M$ and $M_r$ are related by the equality
 \begin{equation}   \label{DzetaRelationRestriction}
 \zeta _{M_r} ( t^r) = \prod\limits _{k=0} ^{r-1} \zeta _M \left( e^{\frac{ 2 \pi i k}{r}} t \right).
 \end{equation}
 In particular, if $M$ has polynomial $\zeta$-quotient
 $P_M(t) = \frac{\zeta _M(t)}{\zeta _{\PP^1 ( \overline{\FF_q})} (t)} = \prod\limits _{j=1} ^d (1 - \omega _jt)$
  of $\deg P_M(t) = d$ then $M_r$ has a polynomial $\zeta$-quotient
  $P_{M_r} (t) = \frac{\zeta _{M_r} (t)}{\zeta _{\PP^1 ( \overline{\FF_{q^r}})} (t)} = \prod\limits _{j=1} ^d (1 - \omega _j ^rt)$
  of $\deg P_{M_r} (t) = d$ and $M$ satisfies the Riemann Hypothesis Analogue with respect to $\PP^1 ( \overline{\FF_q})$ as a $\mathfrak{G}$-module if and only if $M_r$ satisfies the Riemann Hypothesis Analogue with respect to $\PP^1 ( \overline{\FF_q})$ as a $\mathfrak{G}_r$-module.
\end{proposition}

\begin{proof}

According to (1.6) from subsection  V.1 of \cite{St}, for any $m, r \in \NN$ with ${\rm GCD} (m,r) = d \in \NN$ there holds the equality of monic polynomials
\begin{equation}   \label{PolynomialFormula}
\left( x ^{\frac{r}{d}} -1 \right) ^d = \prod\limits _{k=0} ^{r-1} \left(x - e^{\frac{2 \pi i km}{r}} \right).
\end{equation}
Namely, if $m_1 = \frac{m}{d}$, $r_1 = \frac{r}{d}$ then $\frac{m}{r} = \frac{m_1}{r_1}$ and any $0 \leq k \leq r-1$ can be divided by $r_1$ with quotient $0 \leq q \leq d-1$ and remainder $0 \leq s \leq r_1-1$, $k = r_1q + s$.
Thus,
\begin{align*}
\prod\limits _{k=0} ^{r-1} \left(x - e^{\frac{2 \pi i km_1}{r_1}} \right) =
 \prod\limits _{q=0} ^{d-1} \prod\limits _{s=0} ^{r_1-1} \left(x - e^{2 \pi i m_1q} e^{\frac{2 \pi i m_1s}{r_1}} \right) =    \\
  \left[ \prod\limits _{s=0} ^{r_1-1} \left(x - e^{\frac{2 \pi i m_1}{r_1} s} \right) \right] ^d =
\left(x ^{r_1} -1 \right) ^d,
\end{align*}
as far as $e^{\frac{2 \pi i m_1}{r_1}} \in \CC^*$ with ${\rm GCD} ( m_1, r_1) =1$ is a primitive $r_1$-th root of unity.
Substituting $x = t^{-m}$ in (\ref{PolynomialFormula}) and multiplying by $t^{mr}$, one obtains
$$
\left(1 - t^{ \frac{mr}{d}} \right) ^d = \prod\limits _{k=0} ^{ r-1} \left[ 1 - \left( e^{\frac{2 \pi i k}{r}} t \right) ^m \right].
$$
According to Lemma \ref{OrbStab} (iii), any $\mathfrak{G}$-orbit $\nu$ of $\deg \nu = m$ splits in $d$ orbits $\nu = \nu _1 \coprod \ldots \coprod \nu _d$  over  $\mathfrak{G}_r$ of $\deg \nu _j = m_1$ for $\forall 1 \leq j \leq d$.
The contribution of $\nu$    to
$\left[ \prod\limits _{k=0} ^{r-1} \zeta _M \left( e^{\frac{2 \pi i k}{r}} t \right) \right] ^{-1}$ is
$
\prod\limits _{k=0} ^{r-1} \left[ 1 - \left( e^{\frac{2 \pi i k}{r}} t \right) ^m \right] =
\left(1 - t^{r \frac{m}{d}} \right) ^d = \prod\limits _{j=1} ^d \left(1 - t^{r \deg \nu _j} \right)
$
and equals the contribution of $\nu _1 \coprod \ldots \coprod \nu _d$ to $\zeta _{M_r}( t^r) ^{-1}$.
That justifies the equality of power series (\ref{DzetaRelationRestriction}).

For any $\omega  \in \CC^*$ note that
\begin{equation}   \label{BinomialExtraction}
\prod\limits _{k=0} ^{r-1} \left(1 - e^{\frac{2 \pi i k}{r}} \omega  t \right) =
(\omega t) ^r \prod\limits _{k=0} ^{r-1} \left( \frac{1}{\omega t} - e^{\frac{2 \pi i k}{r}} \right) =
(\omega t) ^r \left[ \frac{1}{(\omega t)^r} - 1 \right] =
1 - \omega  ^r t^r.
\end{equation}
Thus, if $P_M(t) = (1-t)(1-qt) \zeta _M(t) = \prod\limits _{j=1} ^d (1 - \omega  _jt) \in \ZZ[t]$ is a polynomial of $\deg P_M(t) =d$ then the application of (\ref{DzetaRelationRestriction}) to $M_r$ and the locally finite $\mathfrak{G}_r$-module $\PP^1 ( \overline{\FF_{q^r}}) = \PP ^1 ( \overline{\FF_q})$, combined with (\ref{BinomialExtraction}) yields
\begin{align*}
P_{M_r} ( t^r) = \frac{\zeta _{M_r} (t^r)}{\zeta _{\PP^1 ( \overline{\FF_q}) _r} (t^r)}  =
\prod\limits _{k=0} ^{r-1} \frac{\zeta _M \left( e^{\frac{2 \pi i k}{r}} t \right)}{\zeta _{\PP^1 ( \overline{\FF_q})} \left( e^{\frac{2 \pi i k}{r}} t \right)} = \prod\limits _{k=0} ^{r-1} P_M \left(  e^{\frac{2 \pi i k}{r}} t \right)  =  \\
\prod\limits _{k=0}^{r-1} \prod\limits _{j=1} ^d \left( 1 - \omega  _j e^{\frac{2 \pi i k}{r}}  t \right) =
 \prod\limits _{j=1} ^d  \prod\limits _{k=0}^{r-1}\left( 1 - \omega  _j e^{\frac{2 \pi i k}{r}}  t \right) =
\prod\limits _{j=1} ^d \left(1 - \omega _j ^r t^r \right).
\end{align*}
Therefore $P_{M_r} (t) = \prod\limits _{j=1} ^d (1 - \omega _j^r t)$ is a polynomial of degree $d$, whose leading coefficient has absolute value
$
\left| {\rm LC} ( P_{M_r} (t) ) \right| = \left| {\rm LC} ( P_M(t)) \right| ^r = \left| a_d \right| ^r.
$
Thus, $|\omega _j| = \sqrt[d]{|a_d|}$ exactly when $\left| \omega _j ^r \right| = \sqrt[d]{\left| a_d \right| ^r}$ and the Riemann Hypothesis Analogues for $M$ and $M_r$ are equivalent to each other.

\end{proof}

%%%%%%%%%%%%%%%%%%%%%%%%%%%%%%%%%%%%%%%%%%%%%%%%%%%%%%%%%%%%%%%%%%%%%%%%%%%%%%%%%%%%%%%%%%%%%%%%%%%%%%%%%%%%
%%%%%%%%%%%%%%%%%%%%%%%%%%%%%%%%%%%%%%%%%%%%%%%%%%%%%%%%%%%%%%%%%%%%%%%%%%%%%%%%%%%%%%%%%%%%%%%%%%%%%%%%%%%%
%%%%%%%%%%%%%%%%%%%%%%%%%%%%%%%%%%%%%%%%%%%%%%%%%%%%%%%%%%%%%%%%%%%%%%%%%%%%%%%%%%%%%%%%%%%%%%%%%%%%%%%%%%%%
%%%%%%%%%%%%%%%%%%%%%%%%%%%%%%%%%%%%%%%%%%%%%%%%%%%%%%%%%%%%%%%%%%%%%%%%%%%%%%%%%%%%%%%%%%%%%%%%%%%%%%%%%%%%
 %%%%%%%%%%%%%%%%%%%%%%%%%%%%%%%%%%%%%%%%%%%%%%%%%%%%%%%%%%%%%%%%%%%%%%%%%%%%%%%%%%%%%%%%%%%%%%%%%%%%%%%%%%%%
%%%%%%%%%%%%%%%%%%%%%%%%%%%%%%%%%%%%%%%%%%%%%%%%%%%%%%%%%%%%%%%%%%%%%%%%%%%%%%%%%%%%%%%%%%%%%%%%%%%%%%%%%%%%
%%%%%%%%%%%%%%%%%%%%%%%%%%%%%%%%%%%%%%%%%%%%%%%%%%%%%%%%%%%%%%%%%%%%%%%%%%%%%%%%%%%%%%%%%%%%%%%%%%%%%%%%%%%%
%%%%%%%%%%%%%%%%%%%%%%%%%%%%%%%%%%%%%%%%%%%%%%%%%%%%%%%%%%%%%%%%%%%%%%%%%%%%%%%%%%%%%%%%%%%%%%%%%%%%%%%%%%%%
%%%%%%%%%%%%%%%%%%%%%%%%%%%%%%%%%%%%%%%%%%%%%%%%%%%%%%%%%%%%%%%%%%%%%%%%%%%%%%%%%%%%%%%%%%%%%%%%%%%%%%%%%%%%

%%%%%%%%%%%%%%%%%%%%%%%%%%%%%%%%%%%%%%%%%%%%%%%%%%%%%%%%%%%%%%%%%%%%%%%%%%%%%%%%%%%%%%%%%%%%%%%%%%%%%%%%%%%%%%%%%%%%%%%%%%%%%%%%%%%%%%%%%%%%%%%%%
\section{  Galois closure of a finite covering of locally finite modules  }

\begin{definition}    \label{FiniteCovering}
A surjective $\mathfrak{G} = {\rm  Gal} ( \overline{\FF_q} / \FF_q)$-equivariant map $\xi : M \rightarrow L$ of locally finite $\mathfrak{G}$-modules is a covering of degree $\deg \xi = k$ if there exist  at most finitely many points $x_1, \ldots , x_s \in L$ such that the  fibres $\xi ^{-1} (x)$ over
 $\forall x \in L \setminus \{ x_1, \ldots , x_s \}$ are of cardinality $\left| \xi ^{-1} (x) \right| = k$ and $\left| \xi ^{-1} ( x_i) \right| < k$ for
 $\forall 1 \leq i \leq s$.
\end{definition}

Note that the inertia map $e_{\xi} : M \rightarrow \NN$ of a covering $\xi : M \rightarrow L$ of $\deg \xi = k$ takes values in $\{ 1, \ldots , k \}$.
This is an immediate consequence of Lemma-Definition \ref{InertiaIndices} (iii), according to which $\xi ^{-1} (x) = \coprod _{y \in S_x} {\rm Orb} _{\mathfrak{G}_{\delta (x)}} (y)$ for $\forall x \in M$, $\delta (x) = \deg {\rm Orb} _{\mathfrak{G}} (x)$,
 $\deg {\rm Orb} _{\mathfrak{G} _{\delta (x)}} (y) = e_{\xi}(y)$, whereas
$$
k  \geq \left| \xi ^{-1} (x) \right| = \sum\limits _{y \in S_x} e_{\xi} (y) \ \ \mbox{  for some    } \ \ e_{\xi}(y) \in \NN.
$$

In order to discuss some examples  of finite coverings  of locally finite $\mathfrak{G}$-modules, let us recall that a rational map $\xi : X \myarrow Y$ or a morphism $\xi : X \rightarrow Y$ of smooth  irreducible quasi-projective  curves  is dominant if its image is Zariski dense in $Y$.
A dominant rational map $\xi : X \myarrow Y$ (or a dominant morphism $\xi : X \rightarrow Y$) is finite if it induces a finite extension
 $\xi ^* : \overline{\FF_q} (Y) \hookrightarrow \overline{\FF_q}(X)$ of the corresponding function fields.
 A finite dominant rational map $\xi : X \myarrow Y$ (or a finite dominant morphism) is separable if the finite extension $\xi ^* : \overline{\FF_q}(Y) \hookrightarrow \overline{\FF_q}(X)$ is separable.

\begin{proposition}   \label{FiniteCoveringOfCurves}
Let $X / \FF_q \subseteq \PP^n ( \overline{\FF_q})$ be a smooth irreducible projective curve, defined over $\FF_q$.
Then:

(i) there exist finite separable surjective morphisms $\xi : X \rightarrow \PP^1 ( \overline{\FF_q})$;

(ii)  any finite separable surjective morphism $\xi : X \rightarrow Y$ onto a projective curve $Y \subseteq \PP^m ( \overline{\FF_q})$ is a finite covering of locally finite $\mathfrak{G}_m = {\rm Gal} ( \overline{\FF_q} / \FF_{q^m})$-modules for some $m \in \NN$.
\end{proposition}

\begin{proof}

(i)  At least one of the standard affine open subsets
$$
U_i = \{ a = [ a_0 : \ldots : a_n] \in \PP^n ( \overline{\FF_q}) \, \vert \, a_i \neq 0 \}, \ \  0 \leq i \leq n
$$
  of $\PP^n ( \overline{\FF_q})$ intersects $X$ in an affine open subset $X \cap U_i \subset U_i \simeq \overline{\FF_q} ^n$.
The function field $\overline{\FF_q}(X) = \overline{\FF_q} ( X \cap U_i) = Q (\overline{\FF_q}  [ X \cap U_i]) $ of $X$ is the quotient  field of the affine coordinate ring $\overline{\FF_q} [X \cap U_i] = \overline{\FF_q} [ x_1, \ldots , x_n] / I(X \cap U_i)$.
In particular, $\overline{\FF_q}(X)$ is a finitely generated extension of $\overline{\FF_q}$.
According to Proposition 1 from  §4 of the Algebraic Preliminaries of \cite{Shaf}, there exist such generators $u, v$ of $\overline{\FF_q}(X) = \overline{\FF_q} (u,v)$ over $\overline{\FF_q}$ that $u$ is transcendent over $\overline{\FF_q}$ and $v$ is separable over $\overline{\FF_q}(u)$.
The inclusion  $\overline{\FF_q} (u) \subseteq \overline{\FF_q} (X)$ corresponds to a finite separable dominant rational map $\xi : X \myarrow \PP^1 ( \overline{\FF_q})$.
Since $X$ is smooth, $\xi : X \rightarrow \PP^1 ( \overline{\FF_q})$ is a morphism (cf.Proposition 3.2.6 from \cite{NX}).
The morphism $\xi$ is non-constant and its image $\xi (X)$ is a closed projective subvariety of $\PP^1 ( \overline{\FF_q})$ by Theorem 2 from §I.5.2 of \cite{Shaf}.
Therefore $\xi (X)   = \PP^1 ( \overline{\FF_q})$ and $\xi : X \rightarrow \PP^1 ( \overline{\FF_q})$ is a surjective finite separable morphism.

(ii)  An arbitrary finite separable surjective morphism $\xi : X \rightarrow Y \subseteq \PP^m ( \overline{\FF_q})$ induces a finite separable extension $\xi ^* : \overline{\FF_q}(Y) \hookrightarrow \overline{\FF_q}(X)$ of the corresponding function fields.
Let $\theta \in \overline{\FF_q} (X)$ be a primitive element of $\xi ^* \overline{\FF_q}(Y) \subseteq \overline{\FF_q} (X)$ with minimal polynomial $\widetilde{h _{\theta}} (x) \in \xi ^* \overline{\FF_q} (Y) [x]$.
If $W_j = \{ b = [ b_0 : \ldots : b_m] \in \PP^m ( \overline{\FF_q}) \, \vert \, b_j \neq 0 \}$  is a standard affine open subset of $\PP^m ( \overline{\FF_q})$ and $Y \cap W_j \subset  Y$  is Zariski open, then the function field $\overline{\FF_q}(Y) = \overline{\FF_q}  (Y \cap W_j) = Q( \overline{\FF_q} [ Y \cap W_j])$ is the quotient field of the affine coordinate ring $\overline{\FF_q} [ Y \cap W_j] = \overline{\FF_q} [ y_1, \ldots , y_m] / I(Y \cap W_j)$   of $Y \cap W_j$.
Multiplying $\widetilde{h_{\theta}} (x)$ by a common multiple of the denominators of its coefficients, one   obtains an element of $\overline{\FF_q} [ Y \cap W_j] [x]$, lifting  to a polynomial $h_{\theta} \in \overline{\FF_q} [ y_1, \ldots , y_m, x]$.
The isomorphism  $\xi ^* : \overline{\FF_q} (Y) ( \theta) \rightarrow \overline{\FF_q} (X)$ of $\overline{\FF_q}$-algebras  corresponds to a birational map
$\psi : X \myarrow \overline{Z}$ in the projective closure $\overline{Z}$ of
$$
Z = \{ (a, a_{m+1}) \in (Y \cap W_j) \times \overline{\FF_q} \, \vert \, h_{\theta} (a, a_{m+1}) =0 \}.
$$
The smoothness of $X$ implies that $\psi : X \rightarrow \overline{Z}$ is a birational morphism.
The composition ${\rm pr} \circ \psi : X \rightarrow Y$ of $\psi$ with the canonical projection ${\rm pr} : Z \rightarrow Y$ on $Y$ induces the extension
 $({\rm pr} \circ \psi) ^* = \psi ^* {\rm pr} ^* : \overline{\FF_q} (Y) \hookrightarrow \overline{\FF_q} (X)$ with
 $ ({\rm pr} \circ \psi) ^* \overline{\FF_q} (Y) = \psi ^* {\rm pr} ^* \overline{\FF_q} (Y) = \xi ^* \overline{\FF_q} (Y)$.
Therefore there is a commutative diagram
$$
\begin{diagram}
\node{X}  \arrow{se,r}{\xi}  \arrow{e,t}{\psi}  \node{\overline{Z}}  \arrow{s,r}{\rm pr}  \\
\node{\mbox{}}   \node{Y}
\end{diagram}
$$
and the fibres of $\xi$ coincide with the fibres  of ${\rm pr}$ with at most finitely many exceptions.
However, $\left| {\rm pr} ^{-1} (b) \right| \leq \deg _x h_{\theta} = k$ with at most finitely many strict inequalities $\left| {\rm pr} ^{-1} (b) \right| < k$ over the points $b \in Y$, for which $h_{\theta} (b, x)$ has multiple roots.
The discriminant $D_x ( h_{\theta} ( y_1, \ldots , y_m, x)) \in \overline{\FF_q} [ y_1, \ldots , y_m]$  of $h_{\theta} \in \overline{\FF_q} [ y_1, \ldots , y_m, x]$ with respect to $x$ has finitely many roots on $Y$, as far as $D_x (h_{\theta}) \not \in I(Y \cap W_j)$ and  the proper subvarieties of $Y \cap W_j$ are finite sets.
Therefore $\left| {\rm pr} ^{-1} (b) \right| = k$ with at most finitely many exceptions.

Towards the $\mathfrak{G}_m = {\rm Gal} ( \overline{\FF_q} / \FF_{q^m})$-equivariance of $\xi$ for some $m \in \NN$, let us cover $Y$ by finitely many affine Zariski open subsets $Y \cap W_j$ and note that the restrictions $\xi : \xi ^{-1} (Y \cap W_j) \rightarrow Y \cap W_j \subset W_j \simeq \overline{\FF_q} ^m$ are given by ordered $m$-tuples $\xi = (\xi _1, \ldots , \xi _m)$ of rational functions $\xi _s : \xi ^{-1} (Y \cap W_j) \rightarrow \overline{\FF_q}$.
The restrictions of $\xi _s$ on the affine open subsets $\xi ^{-1} (Y \cap W_j) \cap U_i$ act as ratios $\frac{f_{i,j,s}}{g_{i,j,s}}$ of polynomials.
If $\FF_{q^{m_{ij}}} \supset \FF_q$ contains the coefficients of $f_{i,j,s}$ and $g_{i,j,s}$ for all $1 \leq s \leq m$ then
 $\xi : \xi ^{-1} (Y \cap W_j) \cap U_i \rightarrow Y$ is $\mathfrak{G}_{m_{ij}} = {\rm Gal} ( \overline{\FF_q} / \FF_{q^{m_{ij}}})$-equivariant.
Denoting by $m$ the least common multiple of $m_{ij}$ for $\forall 0 \leq i \leq n$ and $\forall 0 \leq j \leq m$, one concludes that $\xi : X \rightarrow Y$ is $\mathfrak{G}_m = {\rm Gal} ( \overline{\FF_q} / \FF_{q^m})$-equivariant and, therefore, a covering of $\mathfrak{G}_m$-modules of degree
$k = [ \overline{\FF_q} (X) : \xi ^* \overline{\FF_q} ( \PP^1 ( \overline{\FF_q}))]$.

\end{proof}

A locally finite $\mathfrak{G}$-module $M$ is faithful if the homomorphism $\Psi : \mathfrak{G} \rightarrow {\rm Sym} (M)$, associated with the $\mathfrak{G}$-action is injective.

\begin{lemma}   \label{FaithfulModule}
A locally finite module $M$ over the absolute Galois group $\mathfrak{G} = {\rm Gal} ( \overline{\FF_q} / \FF_q)$ of a finite field $\FF_q$ is faithful if and only if $M$ is an infinite set.
\end{lemma}

\begin{proof}

By the very definition of the homomorphism  $\Psi : \mathfrak{G} \rightarrow {\rm Sym} (M)$, its kernel
 $$
 \ker \Psi = \cap _{x \in M} {\rm Stab} _{\mathfrak{G}} (x)
 $$
  is the intersection of the stabilizers of all the points of $M$.
In the proof of   Lemma \ref{OrbStab} (iii) we have established that $\mathfrak{G}_m \cap \mathfrak{G}_n = \mathfrak{G}_{{\rm LCM} (m,n)}$.
If $M = \{ x_1, \ldots , x_r \}$ is a finite set then the map $\deg {\rm Orb} _{\mathfrak{G}} : M \rightarrow \NN$ has finitely many values $m_1, \ldots , m_{\nu}$, $\nu \leq r$.
As a result, $\ker \Psi = \cap _{j=1} ^{\nu} \mathfrak{G}_{m_j} = \mathfrak{G} _{{\rm LCM} ( m_j \, \vert \, 1 \leq j \leq \nu)} \neq 0$ and $M$ is not a faithful $\mathfrak{G}$-module.

Suppose that $M$ is an infinite locally finite $\mathfrak{G}$-module and $\alpha = ( l_s ( {\rm mod} \, s)) _{s \in \NN} \in \ker \Psi = \cap _{x \in M} {\rm Stab} _{\mathfrak{G}} (x) = \cap _{x \in M} \mathfrak{G}_{\deg {\rm Orb} _{\mathfrak{G}} (x)}$.
Then for any point  $x \in M$ and any  $s \in \NN$, $\deg {\rm Orb} _{\mathfrak{G}} (x)$ divides $l_s$.
The map $\deg {\rm Orb} _{\mathfrak{G}} : M \rightarrow \NN$ has an infinite image, so that any $l_s$ is divisible by infinitely many different natural numbers $\deg {\rm Orb} _{\mathfrak{G}} (x)$, $x \in M$.
That implies $l_s =0$ for $\forall s \in \NN$, whereas $\ker \Psi = \{ 0 \}$ and any infinite locally finite $\mathfrak{G}$-module $M$ is faithful.

\end{proof}

From now on, we consider only infinite locally finite $\mathfrak{G}$-modules $M$ and treat $\mathfrak{G}$ as a subgroup of ${\rm Sym} (M)$.

\begin{lemma-definition}   \label{GaloisCover}
Let $N$ be an infinite locally finite module over $\mathfrak{G} = {\rm Gal} ( \overline{\FF_q} / \FF_q)$ and $H \leq  Z( \mathfrak{G}) :=
\{ h \in {\rm Sym} (N) \, \vert \, h \gamma = \gamma h, \, \forall \gamma \in \mathfrak{G} \} \leq {\rm Sym} (N)$ be a finite subgroup of the centralizer of $\mathfrak{G}$ with $H \cap \mathfrak{G}  = \{ {\rm Id} _N \}$ and ${\rm Stab} _H (x) = \{ {\rm Id} _N \}$ for all but at most  finitely many  $x \in N$.
Then the orbit space $N / H = {\rm Orb} _H (N)$ has a natural structure of a locally finite $\mathfrak{G}$-module and the map $\xi _H : N \rightarrow N/H$, $\xi _H (x) = {\rm Orb} _H(x)$ for $\forall x \in N$ is a finite covering of locally finite $\mathfrak{G}$-modules, which is called an $H$-Galois covering.
 \end{lemma-definition}

\begin{proof}

By assumption, for arbitrary $h \in H$ and $\gamma \in \mathfrak{G}$ there is a commutative diagram
$$
\begin{diagram}
\node{N}  \arrow{s,r}{h}   \arrow{e,t}{\gamma}  \node{N}  \arrow{s,r}{h}  \\
\node{N}  \arrow{e,t}{\gamma}  \node{N}
\end{diagram}.
$$
We define
$$
\mathfrak{G} \times N/H \longrightarrow N/H,
$$
$$
(\gamma , {\rm Orb} _H(x)) \mapsto {\rm Orb} _H (\gamma (x)) \ \ \mbox{  for } \ \ \forall \gamma \in \mathfrak{G}, \ \ \forall x \in N.
$$
Due to ${\rm Orb} _H ( \gamma h (x) = {\rm Orb} _H (h \gamma (x) = {\rm Orb} _H ( \gamma (x))$ for $\forall h \in H$, this map is correctly defined.
The axioms for this $\mathfrak{G}$-action on $N/H$ are immediate consequences of the ones for the $\mathfrak{G}$-action on $N$.
According to $\gamma \xi _H(x) = \gamma {\rm Orb} _H(x) = {\rm Orb} _H ( \gamma (x)) = \xi _H ( \gamma (x))$ for $\forall \gamma \in \mathfrak{G}$,
 $\forall x \in N$, the map $\xi _H : N \rightarrow N /H$ is $\mathfrak{G}$-equivariant.

\end{proof}

The following  proposition explains the etymology of the notion of a finite Galois covering of locally finite modules.

\begin{proposition}      \label{NSCGaloisCoveringOfCurves}
Let $X/ \FF_q \subseteq \PP^n ( \overline{\FF_q})$ be a smooth irreducible projective curve, defined over $\FF_q$ and $\xi : X \rightarrow Y$ be a finite separable surjective morphism onto a smooth projective curve $Y \subseteq \PP^m ( \overline{\FF_q})$.
Then there exist an $H$-Galois covering $\xi _H : X \rightarrow X/H$ of locally finite $\mathfrak{G}_m = {\rm Gal} ( \overline{\FF_q} / \FF_{q^m})$-modules for some $m \in \NN$ and a birational morphism $\psi : Y \rightarrow X/H$, closing the commutative diagram
\begin{equation}   \label{GaloisCoveringDiagram}
\begin{diagram}
\node{X}  \arrow{se,r}{\xi _H}   \arrow{e,t}{\xi}  \node{Y}  \arrow{s,r}{\psi}  \\
\node{\mbox{  }}   \node{X/H}
\end{diagram}
\end{equation}
if and only if $\xi ^* : \overline{\FF_q} (Y)   \hookrightarrow \overline{\FF_q} (X)$ is a Galois extension (i.e., normal and separable) with Galois group
 ${\rm Gal} ( \overline{\FF_q} (X) /  \xi ^* \overline{\FF_q}(Y)) = H$.

 In particular, if $\xi ^* : \overline{\FF_q} (Y) \hookrightarrow \overline{\FF_q}(X)$ is a Galois extension
   then there exist $\mathfrak{G}_s = {\rm Gal} ( \overline{\FF_q} / \FF_{q^s})$-submodules $X_o \subseteq X$, $Y_o \subseteq Y$ with at most finite complements $X \setminus X_o$, $Y \setminus Y_o$  for some $s \in \NN$, such that the restriction
  $\xi \vert _{X_o} = \xi _H \vert _{X_o}  : X_o \rightarrow Y_o$ coincides with the $H$-Galois covering for  $H = {\rm Gal} ( \overline{\FF_q} (X) / \xi ^* \overline{\FF_q} (Y))$.
\end{proposition}

\begin{proof}

Let us suppose that there exist $m \in \NN$,  a finite subgroup $H \leq Z( \mathfrak{G}_m) \leq {\rm Sym} (X)$ with
$H \cap \mathfrak{G}_m = \{ {\rm Id} _X \}$,  ${\rm Stab} _H (x) = \{ {\rm Id} _X \}$ for $\forall x \in X \setminus \{ x_1, \ldots , x_l \}$ and  a birational morphism $\psi : Y \rightarrow X/H$, closing the commutative diagram (\ref{GaloisCoveringDiagram}).
We claim that the function field $\overline{\FF_q} (X/H) = \overline{\FF_q}(X)^H$ of the $H$-quotient $X/H$ coincides with the subfield of the $H$-invariants of $\overline{\FF_q}(X)$.
Indeed, the image of the embedding $\xi _H ^* : \overline{\FF_q} (X/H) \hookrightarrow \overline{\FF_q}(X)$ consists of $H$-invariants, as far as for
 $\forall \rho \in \overline{\FF_q} (X/H)$ and $\forall h \in H$ there holds
$h^* ( \xi _H ^* ( \rho)) =  \rho \circ \xi _H \circ h = \rho \circ \xi _H = \xi _H ^* (\rho)$.
 Conversely, any $\rho \in \overline{\FF_q} (X)^H$ can be viewed as a rational function $\rho : X/H \rightarrow \PP^1 ( \overline{\FF_q})$,
 $\rho ( \xi _H( x)) = \rho (x)$ on $X/H$, since $\rho (x) = h^* ( \rho) (x) = \rho (hx)$ for $\forall x \in X$ and $\forall h \in H$.
 The aforementioned correspondence $\overline{\FF_q} (X)^H \rightarrow \overline{\FF_q} (X/H)$ establishes the invertibility of the embedding
 $\xi _H ^* : \overline{\FF_q} (X/H) \rightarrow \overline{\FF_q}(X) ^H$ and allows to identify $\overline{\FF_q} (X/H) = \overline{\FF_q} (X) ^H$.

On one hand, $|H| = \deg \xi _H = [ \overline{\FF_q} (X) : \overline{\FF_q}(X)^H]$.
On the other hand, $H$ is a subgroup of ${\rm Gal} ( \overline{\FF_q}(X) / \overline{\FF_q} (X)^H)$, so that $|H| \leq \left| {\rm Gal} ( \overline{\FF_q} (X) / \overline{\FF_q} (X)^H) \right| \leq [ \overline{\FF_q}(X) : \overline{\FF_q} (X)^H] = \left| H \right|$ implies that $H = {\rm Gal} ( \overline{\FF_q}(X) / \overline{\FF_q} (X)^H)$ and $\overline{\FF_q}(X) \supseteq \overline{\FF_q}(X)^H$ is a Galois covering.
The commutative diagram (\ref{GaloisCoveringDiagram}) of surjective morphisms induces the commutative diagram
$$
\begin{diagram}
\node{\overline{\FF_q}(X)}  \node{\overline{\FF_q}(Y)}  \arrow{w,t}{\xi^*}  \\
\node{\mbox{  }}  \node{ \overline{\FF_q} (X/H)}  \arrow{nw,r}{\xi ^*_H} \arrow{n,r}{\psi^*}
\end{diagram}
$$
of embeddings of $\overline{\FF_q}$-algebras, so that
$
\overline{\FF_q} (X)^H = \xi _H^* \overline{\FF_q} (X/H) = \xi^* \psi ^* \overline{\FF_q} (X/H) = \xi ^* \overline{\FF_q}(Y)
$
  and $\overline{\FF_q} (X) \supseteq \xi ^* \overline{\FF_q} (Y)$ is a Galois extension.

Conversely, suppose that $\xi ^* : \overline{\FF_q} (Y) \hookrightarrow \overline{\FF_q}(X)$ is a Galois extension with Galois group
 ${\rm Gal} ( \overline{\FF_q}(X) / \xi ^* \overline{\FF_q} (Y)) = H$.
Then according to Galois theory, the $H$-fixed subfield is
$\overline{\FF_q} (X) ^H = \overline{\FF_q} (X) ^{{\rm Gal} ( \overline{\FF_q} (X) / \xi ^* \overline{\FF_q} (Y))} = \xi ^* \overline{\FF_q} (Y)$.
Making use of the isomorphism  of $\overline{\FF_q}$-algebras $\xi _H^* : \overline{\FF_q} (X/H) \rightarrow \overline{\FF_q} (X)^H$,
 one obtains an isomorphism of $\overline{\FF_q}$-algebras $(\xi^*)^{-1} \xi _H ^* : \overline{\FF_q} (X/H) \rightarrow \overline{\FF_q} (Y)$.
 Then there is a   birational map $\psi : Y \myarrow X/H$ with $\psi ^* = ( \xi ^*)^{-1} \xi _H^*$,  which is a morphism $\psi : Y \rightarrow X/H$, due to the smoothness of $Y$.
Moreover, $(\psi \xi) ^* = \xi ^* \psi ^* = \xi _H^*$ implies $\psi \xi = \xi _H$,  i.e.,  the presence of (\ref{GaloisCoveringDiagram}).

As in the proof of Lemma \ref{FiniteCoveringOfCurves}, one covers $X$ by finitely many affine open subsets $X_{i,j} := X \cap U_i \cap \xi _H ^{-1} (W_j) \subset X$, $0 \leq i \leq n$, $0 \leq j \leq m$ with
$$
\xi _H : X_{i,j} \longrightarrow (X/H) \cap W_j \subseteq W_j \simeq \overline{\FF_q} ^m,
 $$
 so that $\xi _H \vert _{X_{i,j}} = ( \xi _{i,j,1}, \ldots , \xi _{i,j,m})$ is given by an ordered $m$-tuple of rational functions $\xi _{i,j,s}$, $1 \leq s \leq m$.
Any such $\xi _{i,j,s} = \frac{f_{i,j,s}}{g_{i,j,s,}}$ is a ratio of polynomials $f_{i,j,s}, g_{i,j,s} \in \overline{\FF_q} [ x_1, \ldots , x_n]$.
If $\FF_{q^m}$ contains the coefficients of $f_{i,j,s}$ and $g_{i,j,s}$ for all $0 \leq i \leq n$, $0 \leq j \leq m$, $1 \leq s \leq m$ then
 $\xi _H : X \rightarrow X/H$ is $\mathfrak{G}_m= {\rm Gal} ( \overline{\FF_q} / \FF_{q^m})$-equivariant, whereas an $H$-Galois covering of locally finite $\mathfrak{G}_m$-modules.

The group $H$ fixes the subfield $\xi ^* \overline{\FF_q}(Y)$, while $\mathfrak{G}_m$ acts on the constant field $\overline{\FF_q}$ of $\xi ^* \overline{\FF_q}(Y)$, so that $H$ centralizes $\mathfrak{G}_m$ and $H \cap \mathfrak{G}_m = \{ {\rm Id} _X \}$.
In order to show that $H$ fixes at most finitely many points on $X$, note that the finite separable extension $\overline{\FF_q} (X) \supseteq \xi ^* \overline{\FF_q} (Y)$ has a primitive element $\theta$, $\overline{\FF_q} (X) = \xi ^* \overline{\FF_q}(Y) ( \theta)$.
Let $Y \cap W_j \subset W_j \simeq \overline{\FF_q} ^m$ be an open affine subset of $Y$, $\widetilde{h_{\theta}} (z) \in \overline{\FF_q} (Y) [z]$ be the minimal polynomial of $\theta$ over $\overline{\FF_q} (Y) = Q ( \overline{\FF_q} [Y \cap W_j])$ and
$h_{\theta} (y_1, \ldots , y_m, z) \in \overline{\FF_q} [ y_1, \ldots , y_m, z]$ be a lifting of the element of $\overline{\FF_q} [ Y \cap W_j] [z]$, obtained from $\widetilde{h_{\theta}} (z)$ by multiplication with a common multiple of the denominators of the coefficients of $\widetilde{h_{\theta}} (z)$.
Then the projective closure $Z$ of
 $$
 Z_o := \{ (a_1, \ldots , a_m, a_{m+1}) \in (Y \cap W_j) \times \overline{\FF_q} \, \vert \, h_{\theta} ( a_1, \ldots , a_m, a_{m+1}) =0 \}
 $$
is birational to $X$ and the fixed points of $h \in H$ on $X$ correspond to the multiple roots of $h_{\theta} ( a_1, \ldots , a_m, z) =0$.
These $a =( a_1, \ldots , a_m) \in Y \cap W_j  \subset W_j \simeq \overline{\FF_q} ^m$ are the zeros of the discriminant $D_z ( h_{\theta} ( y_1, \ldots , y_m, z)) \in \overline{\FF_q} [ y_1, \ldots , y_m]$ of $h_{\theta}$ with respect to $z$.
According to $D_z ( h_{\theta}) \not \in I(Y \cap W_j)$, the affine algebraic set $V ( D_z ( h_{\theta})) \cap (Y \cap W_j)$ is finite.
That shows the presence of at most finitely many fibres of $\xi _H : X \rightarrow X/H$ of cardinality $< |H|$ and $\xi _H$ is an $H$-Galois covering of locally finite $\mathfrak{G}_m$-modules.

We have shown that  if $\xi ^* : \overline{\FF_q} (Y) \hookrightarrow \overline{\FF_q} (X)$ is a finite Galois extension with Galois group
 ${\rm Gal} ( \overline{\FF_q} (X) / \xi ^* \overline{\FF_q} (Y)) = H$, then  there is a commutative diagram (\ref{GaloisCoveringDiagram}) with a birational morphism $\psi : Y \rightarrow X/H$.
 Thus, one can find a Zariski open subset $Z_o \subseteq X/H$,     which is mapped isomorphically   by
   $\psi ^{-1} : Z_o \rightarrow \psi ^{-1} (Z_o) =: Y_o$ onto a Zariski open subset $Y_o \subseteq Y$.
Further, $X_o := \xi ^{-1} (Y_o) = \xi ^{-1} \psi ^{-1} (Z_o) = (\psi \xi) ^{-1} (Z_o) = \xi _H ^{-1} (Z_o)$ is a Zariski open subset of $X$ and the restrictions $\xi \vert _{X_o} = \xi _H \vert _{X_o} : X_o \rightarrow \xi (X_o) = Y_o$ coincide, if we identify $Y_o$ with $\psi (Y_o) = \xi _H(X_o)$.

\end{proof}

Here is an immediate consequence of Lemma-Definition \ref{GaloisCover}.

\begin{corollary}  \label{InertiaIndicesOfGaloisCoverings}
If $\xi _H : N \rightarrow L = N/H$ is a finite $H$-Galois covering of locally finite $\mathfrak{G} = {\rm Gal} ( \overline{\FF_q} / \FF_q)$-modules  then the inertia map $e_{\xi} : N \rightarrow \NN$ is constant on the fibres of $\xi _H$ and an arbitrary inertia index $e \in e_{\xi _H} (N)$ is a natural  divisor of $\left| H \right|$.

In particular, for any $x \in L$ with $\deg {\rm Orb} _{\mathfrak{G}} (x) = \delta (x)$, $\left| \xi _H ^{-1} (x) \right| = \left| H \right|$ and any $y \in \xi _H ^{-1} (x)$ the pull-back $\xi _H ^{-1} {\rm Orb} _{\mathfrak{G}} (x)$ is a disjoint union of $\frac{|H|}{e_{\xi _H} (y)}$ $\mathfrak{G}$-orbits of degree $\delta (x) e_{\xi _H} (y)$ and the fibre $\xi _H ^{-1} (x)$ is a disjoint union of $\frac{|H|}{e_{\xi _H}(y)}$ $\mathfrak{G}_{\delta (x)}$-orbits of degree
 $e_{\xi _H} (y)$.
\end{corollary}

\begin{proof}

According to Lemma-Definition \ref{InertiaIndices} (ii), one has $\xi _H ^{-1} {\rm Orb} _{\mathfrak{G}} (x) = \coprod\limits _{y \in S_x} {\rm Orb} _{\mathfrak{G}} (y)$ with $\deg {\rm Orb} _{\mathfrak{G}} (y) = \delta (x) e_{\xi _H} (y)$ for some subset $S_x \subseteq \xi _H ^{-1} (x)$.
The fibre $\xi _H ^{-1} (x) = {\rm Orb} _H(y)$  is the $H$-orbit of its arbitrary element $y\in \xi _H ^{-1} (x)$, so that there is a subset $S'_x \subseteq H$ of cardinality $\left| S'_x \right| = \left| S_x \right|$, such that
$
\xi _H ^{-1} {\rm Orb} _{\mathfrak{G}} (x) = \coprod\limits _{h \in S'_x} {\rm Orb} _{\mathfrak{G}} (hy).
$
Bearing in mind that $\forall h \in H$ centralizes $\mathfrak{G}$, one observes that
$$
h: {\rm Orb} _{\mathfrak{G}} (y) \longrightarrow h {\rm Orb} _{\mathfrak{G}} (y) = {\rm Orb} _{\mathfrak{G}} (hy), \quad
h ( \gamma (y)) = (h \gamma) (y) \ \ \mbox{  for } \ \ \forall \gamma \in \mathfrak{G}
$$
is a bijective map with inverse
$$
h^{-1} : {\rm Orb} _{\mathfrak{G}} (hy) \longrightarrow {\rm Orb} _{\mathfrak{G}} (y),  \quad
h^{-1} ( \gamma h (y)) = h^{-1} ( h \gamma (y)) = \gamma (y) \ \ \mbox{  for } \ \ \forall \gamma \in \mathfrak{G}.
$$
Therefore $\delta (x) e_{\xi _H} (h(y)) = \deg {\rm Orb} _{\mathfrak{G}} (h(y)) = \deg {\rm Orb} _{\mathfrak{G}} (y) = \delta (x) e_{\xi _H} (y)$ for
 $\forall h \in H$ and the inertia map $e_{\xi _H} : N \rightarrow \NN$ is constant along the $H$-orbits on $N$.
Now (ii) from Lemma-Definition \ref{InertiaIndices} reads as $\xi _H ^{-1} {\rm Orb} _{\mathfrak{G}} (x) = \coprod\limits _{h \in S'_x}  {\rm Orb} _{\mathfrak{G}} (h(y))$ with $\deg {\rm Orb} _{\mathfrak{G}} (h(y)) = \delta (x) e_{\xi _H} (y)$ for $\forall h \in S'_x$.
Similarly,    Lemma-Definition \ref{InertiaIndices} (iii) takes the form $\xi _H ^{-1} (x) = \coprod\limits  _{h \in S'_x} {\rm Orb} _{\mathfrak{G} _{\delta (x)}}  (h(y))$ with $\deg {\rm Orb} _{\mathfrak{G} _{\delta (x)}} (h(y)) = e_{\xi _H} (y)$ for $\forall h \in S'_x$.
Counting the elements of the fibre of $x$, one obtains
$
\left| H \right| = \left| \xi _H ^{-1} (x) \right| = \left| S'_x \right| e_{\xi _H} (y),
$
whereas $|S_x| = |S'_x| = \frac{|H|}{e_{\xi _H}(y)}$.

\end{proof}

Let $\xi _H : N \rightarrow N/H$ be an  $H$-Galois covering of locally finite $\mathfrak{G} = {\rm Gal} ( \overline{\FF_q} / \FF_q)$-modules.
The next proposition describes the ''twists'' of the locally finite $\mathfrak{G}$-action on $N$ and  $N/H$ by the means of an arbitrary element  $h \in H$.

\begin{proposition}   \label{TiwstedProfiniteAction}
Let $N$ be an infinite locally finite module over $\mathfrak{G} = {\rm Gal} ( \overline{\FF_q} / \FF_q)$, $H \leq Z( \mathfrak{G})$ be a finite subgroup of the centralizer of $\mathfrak{G}$ with $H \cap \mathfrak{G} = \{ {\rm Id} _N \}$, ${\rm Stab} _H (x) = \{ {\rm Id} _N \}$ for $\forall x \in N \setminus \{ x_1, \ldots , x_s \}$ and $\varphi = h \Phi _q ^r$ for some $h \in H$, $r \in \NN$.
Then:

(i) $N$ is a locally finite module over the profinite completion $\mathfrak{G} ( \varphi) = \widehat{ \langle \varphi \rangle}$ of the infinite cyclic group $\langle \varphi \rangle \simeq ( \ZZ, +)$;

(ii)  $\mathfrak{G} ( \varphi)$ is a subgroup of $H \times \mathfrak{G} \leq {\rm Sym} (N)$ and the second canonical projection
${\rm pr} _2 : H \times \mathfrak{G} \rightarrow \mathfrak{G}$, ${\rm pr} _2 (h, \gamma) = \gamma$ provides a locally finite action
 \begin{equation}   \label{ProjectedAction}
 \mathfrak{G}( \varphi) \times N/H  \rightarrow N/H,  \quad  \gamma {\rm Orb} _H (x) = {\rm pr} _2 ( \gamma) {\rm Orb} _H (x) \ \ \mbox{  for}  \ \
  \forall \gamma \in \mathfrak{G}(\varphi), \ \ \forall x \in N
 \end{equation}
  of $\mathfrak{G}( \varphi)$ on $N/H$;

(iii) the $H$-Galois covering $\xi _H : N \rightarrow N/H$, $\xi _H (x) = {\rm Orb} _H (x)$ for $\forall x \in N$ is a finite covering of $\mathfrak{G} ( \varphi)$-modules.
\end{proposition}

\begin{proof}

(i) First of all, $\varphi = h \Phi _q^r$ is of infinite order.
Otherwise, for $h$ of order $m$ and $\varphi$ of order $l$, one has ${\rm Id} _N = \varphi ^{ml} = h^m \Phi _q ^{rml} = \Phi _q ^{rml}$ and the Frobenius automorphism $\Phi _q : N \rightarrow N$ turns to be of finite order.
This is an absurd, justifying $\langle \varphi \rangle \simeq (\ZZ, +)$.
Since $H \leq \mathfrak{G} ( \varphi)$ and $H \cap \mathfrak{G} = \{ {\rm Id} _N \}$, the product $H \mathfrak{G} \simeq H \times \mathfrak{G}$  of the subgroups $H$ and $\mathfrak{G}$ of ${\rm Sym} (N)$ is direct.
The subgroup $\langle \varphi \rangle \leq H \mathfrak{G} \leq {\rm Sym} (N)$ acts on $N$, as well as its  profinite completion
 $\mathfrak{G} ( \varphi) = \widehat{ \langle \varphi \rangle}$.
In order to check that the $\mathfrak{G} ( \varphi)$-action on $N$ is locally finite, let us denote by $m$ the order of $h$.
For an arbitrary point $x \in N$ with $\deg {\rm Orb} _{\mathfrak{G}} (x) = \delta$, note that  $\mathfrak{G} ( \varphi ) _{m \delta} := \widehat{ \langle \varphi ^{m \delta} \rangle} = \widehat{ \langle \Phi _q ^{m \delta r} \rangle} = \mathfrak{G}_{m \delta r} \leq \mathfrak{G} _{\delta} = {\rm Stab} _{\mathfrak{G}} (x) \leq {\rm Stab} _{H \times \mathfrak{G}} (x)$.
Therefore $\mathfrak{G}( \varphi) _{m \delta} \leq \mathfrak{G} ( \varphi) \cap {\rm Stab} _{H \times \mathfrak{G}} (x) = {\rm Stab} _{ \mathfrak{G}(\varphi) } (x) \leq \mathfrak{G}( \varphi)$ and
\begin{align*}
\deg {\rm Orb} _{\mathfrak{G} ( \varphi)} (x) = [ \mathfrak{G} ( \varphi) : {\rm Stab} _{\mathfrak{G} ( \varphi)} (x)] =  \\
\frac{[ \mathfrak{G} ( \varphi) : \mathfrak{G} ( \varphi) _{m \delta} ]}{[ {\rm Stab} _{\mathfrak{G} ( \varphi)} (x) : \mathfrak{G} ( \varphi) _{m \delta} ]} =
\frac{m \delta}{[ {\rm Stab} _{\mathfrak{G} ( \varphi)} (x) : \mathfrak{G} ( \varphi) _{m \delta}]} \in \NN.
\end{align*}
That justifies the finiteness of all orbits ${\rm Orb} _{\mathfrak{G} ( \varphi)} (x) \subset N$.
For an arbitrary $n \in \NN$, the $\mathfrak{G} ( \varphi)$-orbits ${\rm Orb} _{\mathfrak{G} (\varphi)} (y)$ of
 $\deg {\rm Orb}  _{\mathfrak{G} ( \varphi)} (y)  = n$ are exactly the ones with stabilizer
 ${\rm Stab} _{\mathfrak{G} ( \varphi)} (y) = \mathfrak{G} ( \varphi) _n$.
 Let $\delta = \deg {\rm Orb} _{\mathfrak{G}} (y)$ and note that $\mathfrak{G} ( \varphi) _{nm} = \widehat{ \langle \varphi ^{nm} \rangle}  = \widehat{ \langle \Phi _q ^{nmr} \rangle} = \mathfrak{G} _{nmr}$ is a subgroup of $\mathfrak{G}$, stabilizing $y$.
Therefore  $\mathfrak{G} _{nmr} \leq {\rm Stab} _{\mathfrak{G}} (y) = \mathfrak{G} _{\delta} = \widehat{ \langle \Phi _q ^{\delta} \rangle}$, whereas
  $\Phi _q ^{nmr} \in \ \widehat{ \langle \Phi _q ^{\delta} \rangle} \cap \langle \Phi _q \rangle = \langle \Phi _q ^{\delta} \rangle$ and
   $\delta$ divides $nmr$.
 For any fixed $n \in \NN$ the natural number $nmr$ has finitely many natural divisors $\delta$.
 Since the $\mathfrak{G}$-action on $N$ is locally finite, for any natural divisor $\delta$ of $nmr$ there are at most finitely many $\mathfrak{G}$-orbits ${\rm Orb} _{\mathfrak{G}} (y) \subset N$ of $\deg {\rm Orb} _{\mathfrak{G}} (y) = \delta$.
 That justifies the presence of at most finitely many $\mathfrak{G} ( \varphi)$-orbits ${\rm Orb} _{\mathfrak{G} ( \varphi)} (y) \subset N$ of $\deg {\rm Orb} _{\mathfrak{G} ( \varphi)} (y) = n$ and shows that the $\mathfrak{G}( \varphi)$-action on $N$ is locally finite.

 (ii) The map (\ref{ProjectedAction}) is correctly defined and coincides with the restriction of the locally finite $\mathfrak{G}$-action on $N/H$ to
 ${\rm pr} _2 \mathfrak{G} ( \varphi) =  \widehat{ {\rm pr} _2 ( \langle h \Phi _q ^r \rangle )} = \widehat{ \langle {\rm pr}_2 (h \Phi _q^r) \rangle} = \widehat{ \langle \Phi _q ^r \rangle} = \mathfrak{G}_r$.
  According to Proposition \ref {Hierarchy}, $N/H$ is a locally finite $\mathfrak{G}_r$-module and, therefore, a locally finite
   $\mathfrak{G} ( \varphi)$-module.

(iii) The $\mathfrak{G}$-equivariance of $\xi _H : N \rightarrow N/H$ implies that
$$
\xi _H ( \gamma x) = \xi _H ( {\rm pr} _1 ( \gamma) {\rm pr} _2 ( \gamma) x) = \xi _H ( {\rm pr} _2 ( \gamma) x) = {\rm pr} _2 ( \gamma) \xi _H (x) \ \ \mbox{ for } \ \ \forall x \in N, \ \ \forall \gamma \in \mathfrak{G} ( \varphi)
$$
with ${\rm pr}_1 ( \gamma) \in H$, ${\rm pr}_2 ( \gamma) \in \mathfrak{G}$.
Thus, $\xi _H$ is $\mathfrak{G}( \varphi)$-equivariant and, therefore, a finite covering of locally finite $\mathfrak{G}( \varphi)$-modules.

\end{proof}

\begin{definition}   \label{GaloisClosure}
Let $\xi : M \rightarrow L$ be a finite covering of locally finite modules over $\mathfrak{G} = {\rm Gal} ( \overline{\FF_q} / \FF_q)$.
We say that $\xi$ has a Galois closure if there exist $m \in \NN$, an infinite locally finite $\mathfrak{G}_m = {\rm Gal} ( \overline{\FF_q} / \FF_{q^m})$-module $N$ and finite subgroups $H_1 \leq H \leq Z ( \mathfrak{G}_m) \leq {\rm Sym} (N)$ of the centralizer of $\mathfrak{G}_m$  with $H \cap \mathfrak{G}_m = \{ {\rm Id} _N \}$, ${\rm Stab} _H (z) = \{ {\rm Id} _N \}$ for $\forall z \in N \setminus \{ z_1, \ldots , z_s \}$, such that $M = N / H_1$, $L = N /H$ and the Galois coverings $\xi _{H_1}$, $\xi _H$, associated with $H_1$, $H$ close the commutative diagram
$$
\begin{diagram}
\node{N} \arrow{e,t}{\xi _{H_1}}  \arrow{se,r}{\xi _H} \node{M = N / H_1}  \arrow{s,r}{\xi}  \\
\node{\mbox{  }}  \node{L = N / H}
\end{diagram}.
$$
\end{definition}

\begin{corollary}    \label{NoehterNormalizationHasGaloisClosure}
Let $X / \FF_q \subseteq \PP^n ( \overline{\FF_q})$ be a smooth irreducible projective curve, defined over $\FF_q$ and $\xi : X \rightarrow \PP^1 ( \overline{\FF_q})$ be a finite separable surjective morphism.
Then there exist $m \in \NN$ and  $\mathfrak{G}_m = {\rm Gal} ( \overline{\FF_q} / \FF_{q^m})$-submodules $X_o \subseteq X$, $L_o \subseteq \PP^1 ( \overline{\FF_q})$ with at most finite complements $X  \setminus X_o$, $\PP^1 ( \overline{\FF_q}) \setminus L_o$, such that the restriction $\xi : X_o \rightarrow L_o$ is a finite covering of $\mathfrak{G}_m$-modules with a Galois closure $\xi _H : Z_o \rightarrow L_o$, supported by a smooth quasi-projective curve $Z_o \subseteq \PP^k ( \overline{\FF_q})$.
\end{corollary}

\begin{proof}

By assumption, $\xi ^* : \overline{\FF_q} ( \PP^1 ( \overline{\FF_q})) \hookrightarrow \overline{\FF_q} (X)$ is a finite separable extension and admits Galois closure $K \supseteq \xi ^* \overline{\FF_q} ( \PP^1 ( \overline{\FF_q}))$.
Since $[K : \xi ^* \PP^1 ( \overline{\FF_q  })] < \infty$, there is a smooth irreducible  projective curve $Z \subseteq \PP^k ( \overline{\FF_q})$ with function field $\overline{\FF_q} (Z) = K$.
The embeddings $\xi ^* \overline{\FF_q} ( \PP^1 ( \overline{\FF_q})) \hookrightarrow \overline{\FF_q} (Z)$, respectively, $\overline{\FF_q} (X) \subseteq  \overline{\FF_q} (Z)$ of $\overline{\FF_q}$-algebras correspond to finite separable dominant rational maps $\xi _o : Z \myarrow \PP^1 ( \overline{\FF_q})$, respectively, $\xi _1 : Z \myarrow  \PP^1 ( \overline{\FF_q})$.
Since $Z$ is smooth, $\xi _o : Z \rightarrow \PP^1 ( \overline{\FF_q})$ and $\xi _1 : Z \rightarrow X$ are finite separable morphisms.
The images $\xi _o (Z)$, respectively, $\xi _1 (Z)$ of the non-constant maps $\xi _o$, $\xi _1$  of the projective curve $Z$ are closed and, therefore, coincide with $\PP^1 ( \overline{\FF_q})$, respectively, $X$.
According to $\xi _o ^* \overline{\FF_q} ( \PP^1 ( \overline{\FF_q})) = \xi ^* \overline{\FF_q} ( \PP^1 ( \overline{\FF_q}))$ and $\xi _1 ^* = {\rm Id} : \overline{\FF_q}(X) \rightarrow \overline{\FF_q}(X)$, there is a commutative diagram
$$
\begin{diagram}
\node{\overline{\FF_q}(Z)}   \node{\overline{\FF_q} (X)}  \arrow{w,t}{\xi _1 ^*}   \\
\node{\mbox{  }}  \node{\overline{\FF_q} ( \PP^1 ( \overline{\FF_q}))}  \arrow{n,r}{\xi ^*}  \arrow{nw,r}{\xi _o ^*}
\end{diagram}
$$
of embeddings of $\overline{\FF_q}$-algebras, which corresponds  to  a  commutative diagram
$$
\begin{diagram}
\node{Z}  \arrow{e,t}{\xi _1}  \arrow{se,r}{\xi _o}  \node{X}  \arrow{s,r}{\xi}  \\
\node{\mbox{  }}  \node{\PP^1 ( \overline{\FF_q})}
\end{diagram}
$$
of finite separable surjective morphisms.
By the very construction of $Z$, the finite separable extensions $\xi _o ^* = \xi ^* : \overline{\FF_q} ( \PP^1 ( \overline{\FF_q})) \hookrightarrow \overline{\FF_q} (Z)$ and $\xi _1 ^* : \overline{\FF_q} (X) \hookrightarrow \overline{\FF_q} (Z)$ are normal, i.e., Galois.
If $H := {\rm Gal} ( \overline{\FF_q} (Z) / \xi _o ^* \overline{\FF_q} ( \PP^1 ( \overline{\FF_q})))$, respectively, $H_1 := {\rm Gal} ( \overline{\FF_q} (Z) / \overline{\FF_q} (X))$ are the corresponding Galois groups then by Proposition \ref{NSCGaloisCoveringOfCurves} there exists $m \in \NN$, such that $H_1 \leq H \leq Z( \mathfrak{G}_m) \leq {\rm Sym} (Z)$, $H \cap \mathfrak{G}_m = \{ {\rm Id} _Z \}$ and ${\rm Stab} _H (z) = \{ {\rm Id} _Z \}$ for $\forall z \in Z \setminus \{ z_1, \ldots , z_l\}$.
Moreover, there  are birational morphisms $\psi : \PP^1 ( \overline{\FF_q}) \rightarrow Z/H$, respectively, $\psi _1 : X \rightarrow Z / H_1$, closing the commutative diagram
$$
\begin{diagram}
\node{ \mbox{  }}  \node{Z/H_1}   \\
\node{Z}  \arrow{s,l}{\xi _H}   \arrow{ne,l}{\xi _{H_1}}  \arrow{e,t}{\xi _1}  \node{X} \arrow{s,r}{\xi}   \arrow{n,r}{\psi_1}  \\
\node{Z/H}  \node{\PP^1 ( \overline{\FF_q})}  \arrow{w,t}{\psi}
\end{diagram}.
$$
Let $L_o \subseteq \PP^1 ( \overline{\FF_q})$ be a Zariski open subset on which $\psi$ restricts to an isomorphism $\psi : L_o \rightarrow \psi (L_o)$ and identify $L_o \equiv \psi ( L_o)$.
The pull back $\xi ^{-1} (L_o) \subseteq X$ is a Zariski open subset and there exists a Zariski open subset $X_o \subseteq \xi ^{-1} (L_o)$ of $X$ with isomorphic restriction  $\psi _1 : X_o \rightarrow \psi _1 (X_o)$.
We identify $X_o \equiv \psi _1 (X_o)$ and note that $\xi (X_o) \subseteq L_o$, in order to identify $\xi (X_o) \equiv \psi \xi ( X_o)$.
The subset $Z_o := \xi _1 ^{-1} (X_o) \subseteq Z$ is Zariski open and its finite coverings $\xi _1 \vert _{Z_o} \equiv \xi _{H_1}  \vert _{Z_o} : Z_o \rightarrow Z_o / H_1 = X_o$, respectively, $\xi _o \vert _{Z_o} \equiv \xi _H \vert _{Z_o} : Z_o \rightarrow Z_o / H = \xi _o (Z_o) = \xi \xi _1 (Z_o) = \xi (X_o)$ coincide.
Bearing in mind that $\xi _o = \xi \xi _1$, one concludes the existence of a commutative diagram
$$
\begin{diagram}
\node{Z_o}  \arrow{se,r}{\xi _H}  \arrow{e,t}{\xi _{H_1}}  \node{X_o}  \arrow{s,r}{\xi} \\
\node{\mbox{  }}  \node{\xi (X_o)}
\end{diagram}
$$
of finite coverings of locally finite $\mathfrak{G}_m = {\rm Gal} ( \overline{\FF_q} / \FF_{q^m})$-modules for some $m \in \NN$.
The Zariski open subsets $X_o \subseteq X$ and $Z_o = \xi _1 ^{-1} (X_o) \subseteq Z$ have at most  finite complements $X \setminus X_o$, $Z \setminus Z_o$.
Therefore $Z = Z_o \coprod  \{ r_1, \ldots , r_s \}$ for some points $r_1, \ldots , r_s \in Z$.
Now, $\PP^1 ( \overline{\FF_q}) = \xi _o (Z) = \xi _o (Z_o) \cup \{ \xi _o ( r_1), \ldots , \xi _o ( r_s) \} = \xi (X_o) \cup \{ \xi _o ( r_1), \ldots , \xi _o ( r_s) \}$ implies that $\xi (X_o)$ has also at most a  finite complement $\PP^1 ( \overline{\FF_q}) \setminus \xi (X_o) \subseteq \{ \xi _o ( r_1), \ldots , \xi _o ( r_s) \}$.

\end{proof}

%%%%%%%%%%%%%%%%%%%%%%%%%%%%%%%%%%%%%%%%%%%%%%%%%%%%%%%%%%%%%%%%%%%%%%%%%%%%%%%%%%%%%%%%%%%%%%%%%%%%%%%%%%%%%%%%%%%%%%%%%%%%%%%%%%%%%%
\section{ Riemann  Hypothesis Analogue for locally finite modules }

\begin{proposition}   \label{RHAReductionToBounds}
The following conditions are equivalent for a locally finite module $M$ over $\mathfrak{G} = {\rm Gal} ( \overline{\FF_q} / \FF_q)$  with a polynomial $\zeta$-quotient $P_M(t) = \frac{\zeta _M(t)}{\zeta _{\PP^1 ( \overline{\FF_q})} (t)} \in  \ZZ[t]$ of $deg P_M(t) = d \in \NN$ with leading coefficient ${\rm LC} ( P_M(t)) = a_d \in \ZZ$ and for  $\lambda := \log _q \sqrt[d]{\left| a_d \right|} \in \RR ^{\geq 0}$:

(i)  $M$ satisfies the Riemann Hypothesis Analogue with respect to  $\PP^1 ( \overline{\FF_q})$ as a $\mathfrak{G}$-module;

(ii) $q^r + 1 - d q^{\lambda r} \leq |M^{\Phi _{q^r}}| \leq q^r +1 + d q^{\lambda r}$ for $\forall r \in \NN$;

(iii) there exist constants $C_1, C_2 \in \RR^{> 0}$,    $\nu, r_1, r_2  \in \NN$, such that
$$
|M ^{\Phi _{q^{\nu r}}}| \leq q^{\nu r} +1 + C_1 q^{\lambda \nu r} \ \ \mbox{  for } \ \   \forall r \in \NN, \ \  r \geq r_1   \ \ \mbox{  and }
$$
$$
|M^{\Phi _{q^{\nu r}}}| \geq q^{\nu r}  +1  - C_2  q^{\lambda \nu r} \ \ \mbox{  for } \ \ \forall r \in \NN, \ \  r \geq r_2.
 $$
\end{proposition}

\begin{proof}

$(i) \Rightarrow (ii)$ If $P_M(t) = \prod\limits _{j=1} ^d (1 - q^{\lambda} e^{i \varphi _j})$ for some $\varphi _j \in [0, 2 \pi)$ then
$$
\left| \PP^1 ( \overline{\FF_q}) ^{\Phi _q ^r} \right| - \left| M ^{\Phi _q ^r} \right| = \sum\limits _{j=1} ^d q^{\lambda r} e^{ i r \varphi _j} \ \  for  \ \ \forall r \in \NN
$$
 by Proposition \ref{NSCPolynomialDzetaQuotient}.
Therefore, the absolute value
$$
\left|    \left| M ^{\Phi _q ^r} \right| - (q^r+1)  \right| =   \left|  \sum\limits _{j=1} ^d q^{\lambda r} e^{i r \varphi _j} \right| \leq \sum\limits _{j=1} ^d \left| q^{\lambda r} e^{i r \varphi _j} \right| = \sum\limits _{j=1} ^d q^{\lambda r} = d q^{\lambda r}
$$
and there holds (ii).

$(ii) \Rightarrow (iii)$  is trivial

$(iii) \Rightarrow (i)$   Note that the power series
$
H(t) := \sum\limits _{j=1} ^d \frac{\omega _j ^{\nu} t}{1 - \omega _j ^{\nu} t}
$
has radius of convergence
$
\rho = \min \left( \frac{1}{\left| \omega _1 \right| ^{\nu}}, \ldots , \frac{1}{\left| \omega _d \right| ^{\nu}} \right),
$
i.e., $H(t) < \infty$ converges for $\forall t \in \CC$ with $|t| < \rho$ and $H(t) = \infty$ diverges for $\forall t \in \CC$ with $|t| > \rho$.
Making use of
$
\frac{1}{1 - \omega _j ^{\nu} t} = \sum\limits _{i=0} ^{\infty} \omega _j ^{\nu i} t^{i}
$
and exchanging the summation order, one represents
$$
H(t) = \sum\limits _{i=0} ^{\infty} \left( \sum\limits _{j=1} ^d \omega _j ^{\nu (i+1)} \right) t^{i+1}.
$$
Let $C := \max ( C_1, C_2)$, $r_0 := \max ( r_1, r_2)$ and note that assumption (iii) implies that
$$
\left| \sum\limits _{j=1} ^d \omega _j ^{\nu r} \right| = \left| \left| M ^{\Phi _q ^r} \right| - (q^{\nu r} +1) \right| \leq C q^{\lambda \nu r} \ \
\mbox{  for  } \ \ \forall r \in \NN, \ \ r \geq r_0,
$$
according to Proposition \ref{NSCPolynomialDzetaQuotient}.
Thus, $\left| \sum\limits _{j=1} ^d \omega _j ^{\nu ( i+1)} \right| \leq C q^{\lambda \nu (i+1)}$    for   $\forall i \in \ZZ,$  $i \geq r_0-1$   and
$$
|H(t)| \leq \sum\limits _{i=0} ^{\infty} \left| \sum\limits _{j=1} ^d \omega  _j ^{\lambda (i+1)} \right| t^{i+1} \leq C \sum\limits _{i=0} ^{\infty} q^{\lambda \nu (i+1)} t^{i+1} = C \sum\limits _{i=0} ^{\infty} (q^{\lambda \nu} t) ^{i+1}.
$$
As a result, $H(t)$ converges for $\forall t \in \CC$ with $|t| < \frac{1}{q^{\lambda \nu}}$, whereas $\frac{1}{q^{\lambda \nu}} \leq \rho \leq \frac{1}{\left| \omega  _j \right| ^{\nu}}$ for $\forall 1 \leq j \leq d$.
The function $f(x) = x^{\nu}$ increases for $x \in \RR ^{\geq 0}$, as far as its derivative $f'(x) = \nu x^{\nu -1} \geq 0$.
Thus, $\left| \omega _j \right| ^{\nu} \leq q^{\lambda \nu}$ implies $\left| \omega _j \right| \leq q^{\lambda}$.
The leading coefficient $a_d = {\rm LC} (P_M(t)) \in \ZZ \setminus \{ 0 \}$ of $P_M(t)$ satisfies the inequality
$$
\left| a_d \right| = \left| \prod\limits _{j=1} ^d ( - \omega  _j) \right| = \prod\limits _{j=1} ^d \left| \omega _j \right| \leq q^{\lambda d} = \left( \sqrt[d] { \left| a_d \right|} \right) ^d = \left| a_d \right|
$$
with equality, so that $\left| \omega _j \right|  = q^{\lambda}$ for $\forall 1 \leq j \leq d$ and $M$ satisfies the Riemann Hypothesis Analogue with respect to $\PP^1 ( \overline{\FF_q})$ as a locally finite module over $\mathfrak{G} = {\rm Gal} ( \overline{\FF_q} / \FF_q)$.

\end{proof}

\begin{corollary}     \label{RHAEquivalentsForTwists}
Let $M$ be an infinite locally finite module over $\mathfrak{G} = {\rm Gal} ( \overline{\FF_q} / \FF_q) = \widehat{ \langle \Phi _q \rangle}$ with 
$\zeta$-function $\zeta _M(t) = \frac{P_M(t)}{(1-t)(1-qt)}$ for a polynomial $P_M(t) \in \ZZ[t]$.
 Suppose that $h \in Z( \mathfrak{G}) \leq {\rm Sym} (M)$ is  an element of order $m$ with $\langle h \rangle \cap \mathfrak{G} = \{ {\rm Id} _M \}$,
  ${\rm Stab} _{\langle h \rangle} (x) = \{ {\rm Id} _M \}$ for $\forall x \in M \setminus \{ x_1, \ldots , x_s \}$, 
   $\mathfrak{G} (h \Phi _q ^n) = \widehat{ \langle h \Phi _q ^n \rangle}$  is  the profinite completion of 
   $(\ZZ, +) \simeq \langle h \Phi _q ^n \rangle \leq {\rm Sym} (M)$ for some $n \in \NN$ and the locally finite $\mathfrak{G}( \varphi)$-module  $M(\varphi)$, supported by $M$ has  $\zeta$-function $\zeta _{M( \varphi)} (t) = \frac{P_{M( \varphi)} (t)}{(1-t)(1-qt)}$ for some polynomial
    $P_{M( \varphi)}(t) \in \ZZ[t]$.  
Then $M$ satisfies the Riemann Hypothesis Analogue with respect to $\PP^1 ( \overline{\FF_q})$ as a locally finite $\mathfrak{G}$-module if and only if $M (\varphi)$ is subject to the Riemann Hypothesis Analogue with respect to $\PP^1 ( \overline{\FF_q})$  as a module over $\mathfrak{G} (h \Phi _q ^n)$.
\end{corollary}

\begin{proof}

By $(i) \Rightarrow (ii)$  from Proposition \ref{RHAReductionToBounds}, the Riemann Hypothesis Analogue over $\mathfrak{G} ( h \Psi _q ^n)$ with $(h \Phi _q ^n) ^m = \Phi _q ^{mn}$ is equivalent to
$$
\left| M ^{\Phi _q ^{mnr}} \right| \leq q^{mnr} +1 + d'q^{\lambda mnr}, \ \
\left| M ^{\Phi _q ^{mnr}} \right| \geq q^{mnr} +1 - d'q^{\lambda mnr} \ \ \mbox{  for  } \ \ \forall r \in \NN
$$
and  $d'= \deg P_{M( \varphi)} (t)$.
According to  $(iii) \Rightarrow (i)$  from Proposition \ref{RHAReductionToBounds}, that is a necessary and sufficient condition for the Riemann Hypothesis Analogue over $\mathfrak{G}$.

\end{proof}

Let $M$ be a locally finite module over the profinite completion  $\mathfrak{G} = \widehat{ \langle \varphi \rangle}$ of an infinite cyclic subgroup
 $(\ZZ, +) \simeq \langle \varphi \rangle \leq {\rm Sym} (M)$.
Consider the graph
$$
\Gamma ^{\varphi} (M) := \{ (x, \varphi (x)) \in M \times M \, \vert \, x \in M \}
$$
 of $\varphi : M \rightarrow M$, the diagonal
$$
\Delta (M) := \{ (x,x) \in M \times M \, \vert \,  x \in M \}
$$
 and the fibre
 $$
 F_2 (M) := \{ (x, x_o) \in M \times M \, \vert \, x \in M \}
 $$
  of the second canonical projection ${\rm pr} _2 : M \times M \rightarrow M$, ${\rm pr} _2 (x,y) = y$ through $x_o \in M$.
  For $D (M)= \Delta (M)$ or $D (M)= F_2 (M)$, let us introduce the intersection number
$$
(\Gamma ^{\varphi} .D) (M)     := \left| \Gamma ^{\varphi} (M) \cap D(M) \right|
$$
 as the cardinality of the set theoretic intersection of $\Gamma ^{\varphi} (M)$ with $D(M)$.

If $\xi : M \rightarrow L$ is a finite covering of locally finite $\mathfrak{G} = \widehat{ \langle \varphi \rangle}$-modules, let us  abbreviate
 $$
 ( \Gamma ^{\varphi} . \Delta) (M \setminus L) := ( \Gamma ^{\varphi} . \Delta) (M) - ( \Gamma ^{\varphi} . \Delta)(L).
 $$

\begin{definition}   \label{OrderOfCovering}
A finite covering $\xi : M \rightarrow L$ of locally finite  $\mathfrak{G} = {\rm Gal} ( \overline{\FF_q} / \FF_q)$-modules    is of order $\lambda \in \RR ^{>0}$ if there exist constants $C  \in \RR ^{>0}$, $r_o \in \NN$, such that
$$
(\Gamma ^{\Phi _q ^r} . \Delta) (M \setminus L) \leq C [ ( \Gamma ^{\Phi _q ^r} . F_2) (L) ] ^{\lambda} \quad  \mbox{  for  } \ \ \forall r \in \NN, \ \ r \geq r_o.
$$
\end{definition}

\begin{definition}    \label{HOrderOfHGaloisCover}
Let $\xi _H : N \rightarrow L$ be an $H$-Galois covering of locally finite modules over  $\mathfrak{G} = {\rm Gal} ( \overline{\FF_q} / \FF_q)$,  for some finite subgroup $H \leq Z( \mathfrak{G}) \leq {\rm Sym} (N)$ with $H \cap \mathfrak{G} = \{ {\rm Id} _N \}$ and ${\rm Stab} _H (z) = \{ {\rm Id} _N \}$ for $\forall z \in N \setminus \{ z_1, \ldots , z_l \}$.
We say that $\xi _H$ is of $H$-order $\lambda \in \RR ^{>0}$ if for $\forall h \in H$ there exist  constants $C(h) \in \RR ^{>0}$, $r(h) \in \NN$, such that
$$
( \Gamma ^{h \Phi _q ^r} . \Delta) (N \setminus L) \leq C(h) [ ( \Gamma ^{h \Phi _q ^r} . F_2) (L) ] ^{\lambda} \quad \mbox{  for  } \ \
 \forall r \in \NN, \ \ r \geq r(h).
$$
\end{definition}

Note that the function $f( \lambda) = a^{\lambda}$ of $\lambda$ is increasing for $a \geq 1$, as far as its derivative $f'( \lambda) = \ln (a) a^{\lambda} \geq 0$.
Thus, if $\xi : M \rightarrow L$ is a finite covering of order $\lambda _1 \in \RR ^{>0}$ and  $\lambda _2 > \lambda _1$, $\lambda _2 \in \RR$ then $\xi : M \rightarrow L$ is of order $\lambda _2$.
Similarly, if $\xi _H : N \rightarrow L$ is an $H$-Galois covering of $H$-order $\lambda _1 \in \RR ^{>0}$ and $\lambda _2 > \lambda _1$, $\lambda _2 \in \RR$  then $\xi _H : N \rightarrow L$ is of $H$-order $\lambda _2$.

\begin{proposition}    \label{CoveringsOfCurvesAreOfOrderHalf}
Let $X / \FF_q \subseteq \PP^n ( \overline{\FF_q})$ be a smooth irreducible projective curve.
Then:

(i) any finite covering $\xi : X_o \rightarrow L_o$ of locally finite $\mathfrak{G}_m = {\rm Gal} ( \overline{\FF_q} / \FF_{q^m})$-submodules $X_o \subseteq X$, $L_o \subseteq \PP^1 ( \overline{\FF_q})$  with $\left| X \setminus X_o \right| < \infty$, $\left| \PP^1 ( \overline{\FF_q}) \setminus L_o \right| < \infty$,
 $m \in \NN$  is of order $\frac{1}{2}$;

(ii)  any $H$-Galois covering $\xi _H : X_o \rightarrow L_o$ of locally finite $\mathfrak{G}_m = {\rm Gal} ( \overline{\FF_q} / \FF_{q^m})$-submodules $X_o \subseteq X$, $L_o \subseteq \PP^1 ( \overline{\FF_q})$  with  $\left| X \setminus X_o \right| < \infty$,
$\left| \PP^1 ( \overline{\FF_q}) \setminus L_o \right| < \infty$, $m \in \NN$
 is of $H$-order $\frac{1}{2}$.
\end{proposition}

\begin{proof}

It suffices to show (ii) for an arbitrary $h \in H$ and to apply  the resulting bound to the case of $h = {\rm Id} _{X_o}$, in order to derive (i).
To this end,    let $X$ be of genus $g$ and
$$
F_1 (X) := \{ (x_o, x) \in X \times X \, \vert \, x \in X \}
$$
 be a fibre of the first canonical projection    ${\rm pr} _1 : X \times X \rightarrow X$, ${\rm pr} _1 (x_1, x_2) = x_1$.
 Note that  $\Gamma ^{\varphi} (X) \subset X \times X$ intersects transversally $\Delta (X) \subset X \times X$, $F_2 (X) \subset X \times X$, so that
 $(\Gamma ^{\varphi}.\Delta) (X)$ and $(\Gamma ^{\varphi} .F_2)(X)$ for $X$ as a locally finite $\mathfrak{G}$-module   coincide with the corresponding intersection numbers of the divisors
 $\Gamma ^{\varphi} (X)$, $\Delta (X)$, $F_2(X)$.
In order to apply a consequence of the Hodge  Index Theorem on $S := X \times X$, let us recall that $\Delta (X) \subset S$ is a smooth irreducible curve of genus $g$, so that $(\Delta. \Delta) (X) = \Delta ^2 (X)   = - (2g-2)$.
Further, $(\Gamma ^{h \Phi _q } . \Delta ) (X)   = \left| X^{ h \Phi _{q^{\rm}}} \right|$, $(\Delta . F_1)  (X)   = (\Delta . F_2) (X)   =1$,
 $(\Gamma ^{h \Phi _q } . F_1) (X) = \left| \{ (x, h \Phi _q  x ) = (y_o, y) \, \vert \, x, y \in X \} \right| =1$ are immediate.
We claim that
$$
\Gamma ^{h \Phi _q^{\rm}} (X) \cap F_2(X) = \{ (x, h \Phi _{q^{\rm}} (x) = (y, y_o) \, \vert \, x, y \in X \}
$$
 is of cardinality $(\Gamma ^{h \Phi _{q^{\rm}}} . F_2) (X) = q $.
More precisely, $h \Phi _q  (x) = y_o$ is equivalent to $\Phi _q  (x) = h^{-1} (y_o)$.
The Frobenius automorphism $\Phi _q $ preserves the standard affine open subsets
$U_i = \{ [ z_0 : \ldots : z_n] \in \PP^n ( \overline{\FF_q}) \, \vert \, z_i \neq 0 \}$ and it suffices to count the solutions of
$\Phi _q  (x) = h^{-1} ( y_o)$ on an affine open neighborhood of $h^{-1} (y_o)$.
Let us fix a local  parameter $t$ on $X \cap U_i$ at a point, different from $h^{-1}(y_o)$, so that $t( h^{-1} (y_o)) = b \in \overline{\FF_q}^*$.
Then view $x = \sum\limits _{i=i_o} ^{\infty} x_i t^{i}$ as a Laurent series with coefficients $x_i \in \overline{\FF_q}$ and reduce the equation
$\Phi _q  (x) = h^{-1} ( y_o)$ to
$\Phi _q \left( \sum\limits _{i=i_o} ^{\infty} x_i t^{i} \right) = \left( \sum\limits _{i=i_o} ^{\infty} x_i t^{i} \right) ^q =
\left( \sum\limits _{i=i_o} ^{\infty} x_i ^q t^{qi} \right) =b.$
Comparing the coefficients of $t$, one concludes that  the solutions $x$ have $x_i =0$ for $\forall i \neq 0$ and $x_0 ^q = b$.
For any fixed $b \in \overline{\FF_q} ^*$ there are exactly $q$ solutions $x_0 \in \overline{\FF_q}$  of $x_0 ^q = b$ and $(\Gamma ^{\Phi _q} . F_2) (X) = q$.
The self-intersection number of the graph $\Gamma ^{h \Phi _q } (X)$ of $h \Phi _{q^{\rm}} : X \rightarrow X$ is computed by the means of the adjunction formula.
Namely, $\Gamma ^{h \Phi _q} (X) \subset X \times X$ is a smooth curve of genus $g$ and the canonical bundle $K_S$ of $S = X \times X$ is numerically equivalent to $(2g-2) ( [ F_1 (X) + F_2(X)]$.
Making use of the adjunction formula
$$
2g-2 = 2g ( \Gamma ^{h \Phi _q} (X)) - 2 = \Gamma ^{h \Phi _q} (X) . ( \Gamma ^{h \Phi _q} (X)  +  K_S) =
\Gamma ^{h \Phi _q} (X)  ^2 + (2g-2) (1+q),
$$
one concludes that $(\Gamma ^{h \Phi _q})^2 (X)  = - q (2g-2).$
The Hodge Index Theorem on $S = X \times X$ asserts that if a divisor $E \subset S$ has vanishing intersection number $E.H =0$ with some ample divisor $H \subset S$ then the self-intersection number $E^2 \leq 0$ of $E$ is non-positive.
For an arbitrary divisor $D \subset S$, the application of the Hodge Index Theorem to the divisors
 $E := D -   [(D.F_2) (X)]  F_1 (X)  -  [(D.F_1)(X) ] F_2(X)  $ and  $H := F_1 (X) + F_2(X)$ yields
\begin{equation}    \label{DivisorFormula}
D^2 \leq 2 [(D.F_1)(X)][(D.F_2)(X)],
\end{equation}
making use of $(F_1)^2 (X) = (F_2)^2 (X) =0$, $(F_1.F_2)(X)=1$ (cf.Proposition 3.9 from \cite{ZetaBook}).
If $D = a \Delta (X) + b \Gamma ^{h \Phi _ q  ^{rm}} (X)$ for some $a, b \in \ZZ$, $b \neq 0$, then (\ref{DivisorFormula}) reads as
$$
f \left( \frac{a}{b} \right) = g \left( \frac{a}{b} \right) ^2 + \left[ q^{rm} +1 - \left| X^{h \Phi _q^{rm}} \right| \right] \left( \frac{a}{b} \right) + q^{rm} g \geq 0 \ \ \mbox{  for } \ \ \forall \frac{a}{b} \in \QQ.
$$
By continuity, the quadratic polynomial
$$
f(z) = g z^2 +  \left[ q^{rm} +1 - \left| X^{h \Phi _q^{rm}} \right| \right] z + q^{rm} g \geq 0
$$
 takes non-negative values for $\forall z \in \RR$.
Therefore, the discriminant
$$
D(f) =  \left[  q^{rm} +1 - \left| X ^{h \Phi _q^{rm}} \right|  \right]  ^2 - 4 q ^{rm} g^2 \leq 0
$$
or
$
\left| \left| X^{h \Phi _q^{rm}} \right|  - (q^{rm} +1) \right| \leq 2 g q^{\frac{rm}{2}}.
$
In particular,
\begin{equation}    \label{ClosedOrderHalf}
\left| X^{h \Phi _q^{rm}} \right| \leq q^{rm} +1 + 2g q^{\frac{rm}{2}} \ \ \mbox{  for } \ \ \forall r \in \NN.
\end{equation}
We  are going to show that
\begin{equation}    \label{OpenOrderHalf}
(\Gamma ^{h \Phi _q^{rm}} . \Delta ) (X_o \setminus L_o) \leq (2 \sqrt{2} g+1) \left[ ( \Gamma ^{h \Phi _q^{rm}} . F_2) (L_o) \right] ^{\frac{1}{2}},
\end{equation}
in order to conclude the proof of the proposition.
To this end, note that the intersection number
 $( \Gamma ^{h \Phi _q^{rm}} . \Delta ) (X_o) = \left| X_o ^{h \Phi _q^{rm}} \right| \leq \left| X^{h \Phi _q^{rm}}  \right|.$
Representing the projective line as a disjoint union  $\PP^1 ( \overline{\FF_q}) = L_o \coprod \{ x_1, \ldots , x_{\alpha} \}$,   one observes that
$\PP^1 ( \FF_{q^{rm}} ) = \PP^1 ( \overline{\FF_q}) ^{\Phi _q^{rm}} =
L_o ^{\Phi _q^{rm}} \coprod \{ x_1, \ldots , x_{\alpha} \} ^{\Phi _q^{rm}} $  and concludes that
$$
(\Gamma ^{h \Phi _q^{rm}} . \Delta) (L_o) = (\Gamma ^{ \Phi _q^{rm}} . \Delta) (L_o) =
 \left| L_o ^{\Phi _q^{rm}} \right| \geq \left| \PP^1 ( \FF _{q^{rm}}) \right| - \alpha = q^{rm} +1 - \alpha.
 $$
 Further,
\begin{equation}   \label{LowerBoundGammaF2}
\begin{split}
  ( \Gamma ^{h \Phi _q^{rm}} . F_2) (L_o) =(\Gamma ^{\Phi _q^{rm}} . F_2) ( \PP^1 ( \overline{\FF_q}) \setminus \{ x_1, \ldots , x_{\alpha} \})= \\
 (\Gamma ^{\Phi _q^{rm}} . F_2) ( \PP^1 ( \overline{\FF_q}) ) - (\Gamma ^{\Phi _q^{rm}} . F_2) ( \{ x_1, \ldots , x_{\alpha} \})  \geq q^{rm} - \alpha
\end{split}
\end{equation}
 by $ (\Gamma ^{\Phi _q^{rm}} . F_2) ( \PP^1 ( \overline{\FF_q}) ) = q^{rm}$  for the smooth irreducible projective curve $\PP^1 ( \overline{\FF_q})$.
 For sufficiently large $r \in \NN$, $r \geq r_o$ one has $\alpha < \frac{1}{2} q^{\frac{rm}{2}}  < \frac{1}{2} q^{rm}$.
  Combining with   (\ref{ClosedOrderHalf}), one  derives that
\begin{align*}
 (\Gamma ^{h \Phi _q^{rm} } . \Delta) (L_o) + (2 \sqrt{2} g +1) [( \Gamma ^{h \Phi _q^{rm}} . F_2) (L_o) ] ^{\frac{1}{2}} \geq  \\
  q^{rm} +1 - \alpha + (2 \sqrt{2} g +1) (q^{rm} - \alpha) ^{\frac{1}{2}} >
  q^{rm} +1 - \frac{1}{2} q^{\frac{rm}{2}} + (2 \sqrt{2} g +1)(q^{rm} - \frac{1}{2} q^{rm}) ^{\frac{1}{2}} =  \\
 q^{rm} +1 + \left( 2g +  \frac{  \sqrt{2}-1}{2} \right) q^{\frac{rm}{2}} >
  q^{rm} +1 + 2g q^{\frac{rm}{2}} \geq \left| X^{h \Phi _q^{rm}} \right| \geq  \left| X_o^{h \Phi _q^{rm}} \right|.
\end{align*}

\end{proof}

\begin{corollary}    \label{CoveringsOf P1HaveOrderHalf}
(i) If $L$ is a $\mathfrak{G} = {\rm Gal} ( \overline{\FF_q} / \FF_q)$-submodule of $\PP^1 ( \overline{\FF_q})$ with at most finite complement
 $\PP^1 ( \overline{\FF_q}) \setminus L$ then any finite covering $\xi : M \rightarrow L$ of locally finite $\mathfrak{G}$-modules is of order $\lambda \geq 1$.

(ii)  If $L$ is a $\mathfrak{G} = {\rm Gal} ( \overline{\FF_q} / \FF_q)$-submodule of $\PP^1 ( \overline{\FF_q})$ with
 $\left| \PP^1 ( \overline{\FF_q}) \setminus L \right| < \infty$ then any finite $H$-Galois covering $\xi _H : N \rightarrow L$ of locally finite $\mathfrak{G}$-modules is of $H$-order $\lambda \geq 1$.
\end{corollary}

\begin{proof}

It suffices to show that for $\forall h \in H$ there exist constants $C \in \RR ^{>0}$, $r_o \in \NN$, such that
$$
(\Gamma ^{h \Phi _{q^r}}. \Delta) (N \setminus L) \leq C [ (\Gamma ^{h \Phi _q^r}. F_2) (L)] \ \ \mbox{  for } \ \ \forall r \in \NN, \ \ r \geq r_o.
$$
To this end, assume that $\left| \PP^1 ( \overline{\FF_q}) \setminus L \right| = \alpha$ and  note  that $(\Gamma ^{h \Phi _q^r} .F_2) (L) \geq q^r - \alpha$ by (\ref{LowerBoundGammaF2}) from the proof of Proposition \ref{CoveringsOfCurvesAreOfOrderHalf}.
For a sufficiently large $r \in \NN$, $r \geq r_o$ one has $\alpha  \leq  \frac{1}{2} q^r$, whereas $(\Gamma ^{h \Phi _q^r} .F_2) (L) \geq \frac{1}{2} q^r$.
On the other hand, $N^{h \Phi _q^r} \subseteq \xi _H ^{-1} (L^{\Phi _q^r})$, as far as for $\forall z \in N$ with $h \Phi _q ^r (z) = z$ there holds $\Phi _q^r \xi _H (z) = \xi _H ( \Phi _q^r z) = \xi _H (h \Phi _q^r z) = \xi _H (z)$.
Therefore $\left| N^{h \Phi _q ^r} \right| \leq \left| \xi _H ^{-1} ( L^{\Phi _q ^r} ) \right| \leq |H| \left| L ^{\Phi _q ^r} \right|$ and
\begin{align*}
(\Gamma ^{h \Phi _q^r} . \Delta)(N \setminus L) =
 \left| N ^{h \Phi _q^r} \right| - \left| L ^{\Phi _q^r} \right| \leq
|H| \left| L ^{\Phi _q^r} \right| - \left| L ^{\Phi _q^r} \right|  =
(|H|-1) \left| L ^{\Phi _q^r} \right| \leq  \\
 (|H|-1) \left| \PP^1 ( \overline{\FF_q}) ^{\Phi _q^r} \right| =
 (|H| -1) (q^r+1) <
  2 (|H|-1) q^r =  \\
   4 (|H| -1) \frac{q^r}{2} \leq
  4 ( |H| -1) (\Gamma ^{h \Phi _q ^r} . F_2) (L).
\end{align*}

\end{proof}

The following simple lemma is crucial for the proof of  the  main Theorem \ref{MainTheorem}.

\begin{lemma}    \label{FixedPointsCountByGaloisCover}
Let $\xi _H : N \rightarrow L$ be an $H$-Galois covering of infinite locally finite modules over $\mathfrak{G} = {\rm Gal} ( \overline{\FF_q} / \FF_q)$ for some finite subgroup $H \leq Z( \mathfrak{G})$ with $H \cap \mathfrak{G} = \{ {\rm Id} _N \}$ and '${\rm Stab} _H (y) = \{ {\rm Id} _N \}$ for $\forall y \in N \setminus \{ y_1, \ldots , y_l \}$.
Then
$$
\sum\limits _{h \in H}  \left|  N ^{h \Phi _q} \right| = \left| H \right| \left| L^{\Phi _q} \right|.
$$
\end{lemma}

\begin{proof}

Let $x \in L^{\Phi _q}$, $y \in \xi _H ^{-1} (x)$, $\varepsilon (y) = \left|  {\rm Stab} _H (y) \right|$ and $\nu (y)$ be the number of times $y$ is counted in $\sum\limits _{h \in H} \left| N ^{ h \Phi_q} \right|.$
Then $x$ has $\left| \xi _H ^{-1} (x) \right| = \left| {\rm Orb} _H (y) \right| = [H : {\rm Stab} _H (y)] = \frac{|H|}{\varepsilon (y)}$ pre-images in $N$.
 It suffices to show that $\varepsilon (y) = \nu (y)$ for $\forall y \in \xi _H^{-1}(x)$, in order to conclude that the contribution of $x$ to $\sum\limits _{h \in H} \left| N ^{ h \Phi_q} \right|$ equals $|H|$ and to prove the lemma.
To this end,   note   that $y \in N ^{h \Phi _q}$ exactly when $h \Phi _q = \Phi _q h \in {\rm Stab} _{\mathfrak{G} \times H} (y) \cap \Phi _q H$.
 We claim that ${\rm Stab} _{\mathfrak{G} \times H} (y) \cap \Phi _q H = \Phi _q h_o {\rm Stab} _H (y)$ for some $h_o \in H$ with $h_o \Phi _q (y) = y$.
The inclusion ${\rm Stab} _{\mathfrak{G} \times H} (y) \cap \Phi _q H \supseteq \Phi _q h_o {\rm Stab} _H (y)$ is clear.
  For the opposite inclusion, if $\Phi _q h_1 \in {\rm Stab} _{\mathfrak{G} \times H} (y) \cap \Phi _q H$ then $h_o ^{-1} h_1 =(\Phi _q h_o)^{-1}(\Phi _qh_1) \in {\rm Stab} _{\mathfrak{G} \times H} (y) \cap H = {\rm Stab} _H (y)$, whereas $\Phi _q h_1 = \Phi _q h_o (h_o ^{-1} h_1) \in \Phi _q h_o {\rm Stab} _H (y)$ and ${\rm Stab} _{\mathfrak{G} \times H} (y) \cap \Phi _q H \subseteq \Phi _q h_o {\rm Stab} _H (y)$.
That justifies ${\rm Stab} _{\mathfrak{G} \times H} (y) \cap \Phi _q H = \Phi _q h_o {\rm Stab} _H (y)$ and implies that
$$
\nu (y) = \left| {\rm Stab} _{\mathfrak{G} \times H} (y) \cap \Phi _q H \right| = \left| \Phi _q h_o {\rm Stab} _H (y) \right| =
 \left| {\rm Stab} _H (y) \right| = \varepsilon (y).
$$

\end{proof}

Here is the main result of the article.

\begin{theorem}     \label{MainTheorem}
Let $M$ be an infinite locally finite module over $\mathfrak{G} = {\rm Gal} ( \overline{\FF_q} / \FF_q)$ with a polynomial $\zeta$-quotient
$P_M(t) = \frac{\zeta _M(t)}{\zeta _{\PP^1 ( \overline{\FF_q})} (t)} = \sum\limits _{j=0} ^d a_j t^{j} \in \ZZ[t]$ of $\deg P_M(t) = d \in \NN$ with  leading coefficient ${\rm LC} ( P_M(t)) = a_d \in \ZZ \setminus \{ 0, \pm 1 \}$ and $\lambda := \log _q \sqrt[d]{\left| a_d \right|} \in \RR ^{>0}$.
Suppose that there exist $m \in \NN$ and $\mathfrak{G}_m = {\rm Gal} ( \overline{\FF_q} / \FF_{q^m})$-submodules $M_o \subseteq M$, $L_o \subseteq \PP^1 ( \overline{\FF_q})$ with $\left| M \setminus M_o \right| < \infty$, $\left| \PP^1 ( \overline{\FF_q}) \setminus L_o \right| < \infty$, which admit a finite covering $\xi : M_o \rightarrow L_o$ of $\mathfrak{G}_m$-modules with a Galois closure $\xi _H : N_o \rightarrow L_o$.

(i) If $\lambda \geq 1$ then $M$ satisfies the Riemann Hypothesis Analogue with respect to the projective line $\PP^1 ( \overline{\FF_q})$ as a $\mathfrak{G}$-module.

(ii) If $\lambda < 1$, $\xi $ is of order $\lambda$ and $\xi _H$ is of $H$-order $\lambda$ then $M$ satisfies the Riemann Hypothesis  Analogue with respect to the projective line $\PP^1 ( \overline{\FF_q})$ as a $\mathfrak{G}$-module.
\end{theorem}

\begin{proof}

According to Corollary \ref{CoveringsOf P1HaveOrderHalf}, an arbitrary finite covering  $\xi: M_o \rightarrow L_o$ is of order $\geq 1$ and an arbitrary finite $H$-Galois covering $\xi _H : N_o \rightarrow L_o$ is of $H$-order $\geq 1$.
That is why, it suffices to show that the existence of a finite covering $\xi : M_o \rightarrow L_o$ of $\mathfrak{G}_m$-modules of order $\lambda := \log _1 \sqrt[d]{\left| a_d \right|} \in \RR ^{>0}$ with a Galois closure $\xi _H : N_o \rightarrow L_o$ of $H$-order $\lambda$ suffices for a locally  finite $\mathfrak{G}$-module $M$ with $P_M(t) = \frac{\zeta _M(t)}{\zeta _{\PP^1 ( \overline{\FF_q})} (t)} = \sum\limits _{j=0} ^d a_j t^{j} \in \ZZ[t]$ to satisfy the Riemann Hypothesis Analogue with respect to $\PP^1 ( \overline{\FF_q})$ as a $\mathfrak{G}$-module.

By assumption, there exist constants  $C_o \in \RR ^{>0}$, $r_o \in \NN$, such that
\begin{align*}
\left| M_o ^{\Phi _q ^{mr}} \right| \leq \left| L_o ^{\Phi _q ^{mr}} \right| + C_o [( \Gamma ^{\Phi _q^{mr}} . F_2) (L_o) ] ^{\lambda} \leq
\left| \PP^1 ( \overline{\FF_q}) ^{\Phi _q^{mr}} \right| + C_o [(\Gamma ^{\Phi _q ^{mr}} . F_2) ( \PP^1 ( \overline{\FF_q}))] ^{\lambda} \leq  \\
q^{mr} +1 + C_o q^{\lambda mr} \ \ \mbox{  for  } \ \ \forall r \in \NN, \, \forall r \geq r_o.
\end{align*}
If $M \setminus M_o = \{ y_1, \ldots , y_{\beta} \}$ then $\left| M ^{\Phi _q^{mr}} \right| - \left| M_o ^{\Phi _q ^{mr}} \right| = \left| (M \setminus M_o) ^{\Phi _q ^{mr}} \right| = \left| \{ y_1, \ldots , y_{\beta} \} ^{\Phi _q ^{mr}} \right| \leq \beta$, whereas $\left| M ^{\Phi _q ^{mr}} \right| \leq \left| M_o ^{\Phi _q^{mr}} \right| + \beta \leq q^{mr} +1 + C_o q^{\lambda mr} + \beta.$
There exists $\rho _o \in \NN$, $\rho _o \geq r_o$ with $\beta \leq q^{\lambda mr}$ for $\forall r \in \NN$, $r \geq \rho _o$.
Introducing $C_1 := C_o +1$, one obtains
$$
\left| M ^{\Phi _q ^{mr}} \right| \leq q^{mr} +1 + C_1 q^{\lambda mr} \ \ \mbox{  for } \ \ \forall r \in \NN, \, r \geq \rho _o.
$$
Note that Lemma \ref{FixedPointsCountByGaloisCover} provides
\begin{equation}   \label{FixedPointsCountOnLByGaloisCover}
\sum\limits _{h \in H} \left| N_o ^{h \Phi _q ^{\nu r}} \right| = |H| \left| L_o ^{\Phi _q ^{\nu r}} \right|
\end{equation}
and
\begin{equation}   \label{FixedPointsCountOnMByGaloisCover}
\sum\limits _{h \in H_1} \left| N_o ^{h \Phi _q ^{\nu r}} \right| = |H_1| \left| M_o ^{\Phi _q ^{\nu r}} \right| \ \ \mbox{  for } \ \ \forall r \in \NN
\end{equation}
and some multiple $\nu \in m \NN$ of $m$, for which
$$
\begin{diagram}
\node{N_o}  \arrow{se,r}{\xi _H} \arrow{e,t}{\xi _{H_1}} \node{M_o}  \arrow{s,r}{\xi}  \\
\node{\mbox{  }} \node{L_o}
\end{diagram}
$$
is a commutative diagram of finite coverings of $\mathfrak{G}_{\nu} = {\rm Gal} ( \overline{\FF_q} / \FF_{q^{\nu}} )$-modules.
 If   $r_1 \in \NN$, $r_1 \geq \frac{\nu }{m} \rho _o$, then upper bound $\left| M^{\Phi _q ^{\nu r}} \right| \leq q^{\nu r} +1 + C_1 q^{\lambda \nu r}$ holds for  any  $r \in \NN$ with  $r \geq r_1$.
Towards a lower bound on $\left| M ^{\Phi _q ^{\nu r}} \right|$, one  makes use of (\ref{FixedPointsCountOnLByGaloisCover}), (\ref{FixedPointsCountOnMByGaloisCover}), in order to note that
\begin{align*}
\left| H_1 \right| \left| M ^{\Phi _q^{\nu r}} \right| \geq \left| H_1 \right| \left| M_o ^{\Phi _q^{\nu r}} \right| =
\sum\limits _{h \in H_1} \left| N_o ^{h \Phi _q ^{\nu r}} \right| + |H| \left| L_o ^{\Phi _q ^{\nu r}} \right| - \sum\limits _{h \in H}  \left| N_o ^{h \Phi _q ^{\nu r}} \right| =  \\
|H| \left| L_o ^{\Phi _q ^{\nu r}} \right| - \sum\limits _{h \in H \setminus H_1} \left| N_o ^{ h \Phi _{q^{\nu r}}} \right| \ \ \mbox{  for } \ \ \forall r \in \NN.
\end{align*}
If $\PP^1 ( \overline{\FF_q}) \setminus L_o = \{ x_1, \ldots , x_{\alpha} \}$ then there  is an upper bound
 $\left| \PP^1 ( \overline{\FF_q} ) ^{\Phi _q ^{\nu r}} \right| - \left| L_o ^{\Phi _q ^{\nu r}} \right| = \left| (\PP^1 ( \overline{\FF_q}) \setminus L_o) ^{\Phi _q ^{\nu r}} \right| = \left| \{ x_1, \ldots , x_{\alpha} \} ^{\Phi _q ^{\nu r}}  \right| \leq \alpha$, whereas $\left| L_o ^{\Phi _q ^{\nu r}} \right| \geq q^{\nu r} +1 - \alpha$.
There exists $\rho \in \NN$ with $\alpha \leq q^{\lambda \nu r}$ for $\forall r \in \NN$, $r \geq \rho$.
Thus, $\left| L_o ^{\Phi _q ^{\nu r}} \right| \geq q^{\nu r} +1 - q^{\lambda \nu r}$ for $\forall r \geq \rho$.
By assumption, $\xi _H : N_o \rightarrow L_o$ is of $H$-order $\lambda$, so that for $\forall h \in H$  there are constants $C(h) \in \RR ^{>0}$, $r(h) \in \NN$ with
\begin{align*}
\left| N_o ^{h \Phi _q ^{\nu r}} \right| \leq \left| L_o ^{h \Phi _q ^{\nu r}} \right| + C(h) [( \Gamma ^{h \Phi _q ^{\nu r}} . F_2) (L_o) ] ^{\lambda} \leq
\left| L_o ^{h \Phi _q ^{\nu r}} \right| + C(h) q^{\lambda \nu r} \,  \mbox{  for } \,  \forall r \in \NN, \, r \geq r(h)
\end{align*}
by $ ( \Gamma ^{h \Phi _q ^{\nu r}} . F_2) (L_o) = ( \Gamma ^{ \Phi _q ^{\nu r}} . F_2) (L_o) \leq
 ( \Gamma ^{ \Phi _q ^{\nu r}} . F_2) (\PP^1 ( \overline{\FF_q}) ) = \left| \PP^1 ( \overline{\FF_q} ) ^{\Phi _q ^{\nu r}} \right|  =
  \left| \PP^1 ( \FF_{q^{\nu r}}) \right| =  q^{\nu r}$.
If $r_2 := \max \{ \rho, r(h) \, \vert \, h \in H \}$ and $C':= \max \{ C(h) \, \vert \, h \in H \} \in \RR ^{>0}$ then
$$
- \left| N_o ^{h \Phi _q ^{\nu r}} \right| \geq - \left| L_o ^{\Phi _q ^{\nu r}} \right| - C'q^{\lambda \nu r} \ \ \mbox{  for } \ \ \forall r \in \NN, \ \ r \geq r_2.
$$
As a result,
\begin{align*}
\left| H_1 \right| \left| M ^{\Phi _q ^{\nu r}} \right| \geq \left| H \right| \left| L_o ^{\Phi _q ^{\nu r}} \right| - ( \left| H \right| - \left| H_1 \right| ) \left| L_o ^{\Phi _q ^{\nu r}} \right| - ( \left| H \right| - \left| H_1 \right| ) C' q^{\lambda \nu r}  =  \\
 \left| H_1 \right| \left| L_o ^{\Phi _q ^{\nu r}} \right| - ( \left| H \right| - \left| H_1 \right| ) C'q^{\lambda \nu r} \ \ \mbox{  for } \ \ \forall r \in \NN, \ \ r \geq r_2.
\end{align*}
Dividing by $\left| H_1 \right| \in \NN$, one obtains
$$
\left| M^{ \Phi _q ^{\nu r}} \right| \ge \left| L_o ^{\Phi _q ^{\nu r}} \right| - \left( \frac{|H| - |H_1|}{|H_1|} \right) C'q^{\lambda \nu r} \geq q^{\nu r} +1 - \left[ \left( \frac{|H| - |H_1|}{|H_1|} \right)  C'+1 \right] q^{\lambda \nu r}.
$$
Thus, for $C_2 := \left( \frac{|H| - |H_1|}{|H_1|} \right) C' + 1 \in \RR ^{>0}$ one has $\left| M ^{\Phi _q ^{\nu r}} \right| \geq q^{\nu r} + 1 - C_2 q^{\lambda \nu r}$ for $\forall r \in \NN$, $r \geq r_2$.
In such a way, we have shown that if there is a finite covering $\xi : M_o \rightarrow L_o$ of locally finite
$\mathfrak{G}_m = {\rm Gal} ( \overline{\FF_q} / \FF_{q^m})$-modules of order $\lambda := \log _q \sqrt[d]{\left| a_d \right|} \in \RR ^{>0}$  with Galois closure $\xi _H : N_o \rightarrow L_o$ of $H$-order $\lambda$ for $\mathfrak{G}_m$-submodules $M_o \subseteq M$, $L_o \subseteq \PP^1 ( \overline{\FF_q})$ with $\left| M \setminus M_o \right| < \infty$, $\left| \PP^1 ( \overline{\FF_q} ) \setminus L_o \right| < \infty$, then $M$ fulfils condition (iii) from Proposition \ref{RHAReductionToBounds} and $M$ satisfies the Riemann Hypothesis Analogue with respect to the projective line $\PP^1 ( \overline{\FF_q})$ as a locally finite module over $\mathfrak{G}$.

\end{proof}

Here is an example of a locally finite $  \widehat{\ZZ}$-module $M$, which satisfies the Riemann Hypothesis Analogue with respect to $\PP^1 ( \overline{\FF_q})$ and is not isomorphic (as a $  \widehat{\ZZ}$-module) to a smooth irreducible projective variety.

\begin{proposition}    \label{NonProjectiveExample}
For arbitrary natural numbers $d, m \in \NN$, $d \leq m$ and an arbitrary field $\FF_q$, there is a locally finite $\mathfrak{G}_m = {\rm Gal} ( \overline{\FF_q} / \FF_{q^m})$-module $M$ with $\zeta$-function
$$
\zeta _M(t) = \frac{(1-qt) ^d}{(1-t)(1-q^mt)}.
$$
In particular, $M$ satisfies the Riemann Hypothesis Analogue with respect to the projective line $\PP^1 ( \overline{\FF_q})$ and for $m \geq 3$ there is no smooth irreducible algebraic  variety $Y / \FF_{q^m} \subseteq \PP^r ( \overline{\FF_q})$, which is isomorphic to $M$ as a locally finite module over $\mathfrak{G}_m$.
\end{proposition}

\begin{proof}

Let
$$
Z := \cup _{i=1} ^d V( x_1, \ldots , x_{i-1}, x_{i+1}, \ldots , x_n) \cup \{ z_1, \ldots , z_{d-2} \} \subset \overline{\FF_q}^m
$$
 be the union of the affine lines $V( x_1, \ldots , x_{i-1}, x_{i+1}, \ldots , x_n) = 0^{i-1} \times \overline{\FF_q} \times 0^{m-i}$ with the points $z_j := 0^{j-1} \times 1 \times 0^{m-1-j} \times 1$, $1 \leq j \leq d-2$ and
 $$
 U := \overline{\FF_q}^m \setminus Z
 $$
  be the complement of $Z$.
Then $Z$ and $U$ are invariant under the action of the absolute Galois group ${\rm Gal} ( \overline{\FF_q} / \FF_q)$ of $\FF_q$  with topological generator $\Phi _q$, $\Phi _q ( a_1, \ldots , a_m) = (a_1 ^q, \ldots , a_m ^q)$ for $\forall (a_1, \ldots , a_m) \in \overline{\FF_q}^m$.
The $\mathfrak{G}$-action on $U$ is locally finite as a restriction of the locally finite $\mathfrak{G}$-action on $\overline{\FF_q}^m$.
In order to compute the $\zeta$-function of $U$ as a $\mathfrak{G}$-module, note that
$$
Z = \left( \coprod\limits _{i=1} ^d 0^{i-1} \times \overline{\FF_q}^* \times 0^{m-i} \right) \coprod \left \{ z_1, \ldots , z_{d-2} , 0^m \right \}
$$
is a disjoint union of $d$ punctured affine lines with $d-1$ points, fixed by $\mathfrak{G}$.
Therefore
$$
Z^{\Phi _q ^r} = \left( \coprod\limits _{i=1} ^d 0^{i-1} \times \FF_{q^r} ^* \times 0^{m-i} \right) \coprod \{ z_1, \ldots , z_{d-2} , 0^m \} \ \ \mbox{  for } \ \ \forall r \in \NN
$$
and $U ^{\Phi _q ^r} = ( \overline{\FF_q}^m) ^{\Phi _q ^r} \setminus Z^{\Phi _q ^r} = \FF_{q^r} ^m \setminus Z^{\Phi _q ^r}$ is of cardinality
 $$
 \left| U ^{\Phi _q ^r} \right| = q^{rm}  - [d(q^r-1) + d-1] = q^{rm} - dq^r +1.
 $$
 Making use of (\ref{LogarithmExpansion}), one computes that
\begin{align*}
\log \zeta _U(t) = \sum\limits _{r=1} ^{\infty} \left| U^{\Phi _q ^r} \right| \frac{t^r}{r} =
 \sum\limits _{r=1} ^{\infty} \left( \frac{q^m t}{r} \right) ^r - d \left( \sum\limits _{r=1} ^{\infty} \frac{(qt)^r}{r} \right)
 + \left( \sum\limits _{r=1} ^{\infty} \frac{t^r}{r} \right) =  \\
 - \log  (1 - q^mt) + d \log (1-qt) - \log (1-t) =
 \log \frac{(1-qt)^d}{(1-t)(1-q^mt)}.
 \end{align*}
Thus, the $\zeta$-function of $U$ as a locally finite module over  $\mathfrak{G} = {\rm Gal} ( \overline{\FF_q} / \FF_q)$ is
$$
\zeta _U(t) = \frac{(1-qt)^d}{(1-t)(1-q^mt)}.
$$
Now, we define such a $\mathfrak{G}_m = {\rm Gal} ( \overline{\FF_q} / \FF_{q^m})$-action on $U$ which turns  the set $M := U$ into a locally finite $\mathfrak{G}_m$-module  with $\zeta _M(t) = \zeta _U(t)$.
More precisely, let
$$
\Phi _{q^m} : M \longrightarrow M,   \quad
\Phi _{q^m} (a) := (a_1 ^q, \ldots , a_m ^q) \ \ \mbox{  for } \ \ \forall a \in M.
$$
The resulting action $\langle \Phi _{q^m} \rangle \times M \rightarrow M$ is locally finite, as far as its orbits coincide with the orbits of the original $\mathfrak{G} = {\rm Gal} ( \overline{\FF_q} / \FF_q)$-action on $U$.
Moreover, it has a unique extension to the profinite completion $\mathfrak{G}_m = \widehat{ \langle \Phi _{q^m} \rangle}$ by the rule
$$
\mathfrak{G}_m \times M \longrightarrow M,
$$
$$
\left( \Phi _{q^m} ^{l_s ({\rm mod}   s)} \right) _{s \in \NN} (a) := \left( \Phi _{q^m} \right) ^{l_{\delta}} (a)   \mbox{  for }
\forall a \in M    \mbox{   of }    \deg {\rm Orb} _{\langle \Phi _{q^m} \rangle} (a) = \deg {\rm Orb} _{\mathfrak{G}_m} (a) = \delta.
$$
According to  the coincidence   $M ^{(\Phi _{q^m})^r} = U^{\Phi _q^r}$ of the fixed point sets, the $\zeta$-function of $M$ as a $\mathfrak{G}_m$-module is
$$
\zeta _M(t) = \exp \left( \sum\limits _{r=1} ^{\infty} \left| M ^{(\Phi _{q^m}) ^r} \right| \frac{t^r}{r} \right) =
\exp \left( \sum\limits _{r=1} ^{\infty} \left| U^{\Phi _q ^r} \right| \frac{t^r}{r} \right) = \zeta _U(t).
$$

By the very definition, $M$ satisfies the Riemann Hypothesis Analogue with respect to $\PP^1 ( \overline{\FF_q})$ as a module over $\mathfrak{G}_m$.
Let us assume the existence of a smooth irreducible projective variety $Y / \FF_{q^m} \subseteq \PP^r ( \overline{\FF_q})$, which is isomorphic to $M$ as a module over $\mathfrak{G}_m = {\rm Gal} ( \overline{\FF_q} / \FF_{q^m})$.
Then by  the Riemann Hypothesis Analogue for $Y$ of dimension $\dim Y = n$ (cf.Conjecture 2.17 from \cite{ZetaBook}), the $\zeta$-function
$$
\zeta _M(t) = \zeta _Y(t) = \frac{\prod\limits _{s=1} ^n \prod\limits _{j_s} (1 - q^{\frac{ms}{2}} e^{i \varphi _{j_s}} t)}{(1-t) \left[  \prod\limits _{s=1} ^{n-1} \prod\limits _{j_s} (1 - q^{ms} e^{i \varphi _{j_s}} t) \right] (1 - q^{mn}t)}.
$$

The specific form of $\zeta _M(t)$ requires $n=1$ and $m=2$.
That shows the non-existence of a smooth irreducible $Y / \FF_{q^m} \subseteq \PP^r ( \overline{\FF_q})$ with $m \geq 3$, which is isomorphic to $M$ as a locally finite module over $\mathfrak{G}_m$.

\end{proof}

The local Weil $\zeta$-function of a smooth irreducible curve $X / \FF_q \subseteq \PP^n ( \overline{\FF_q})$ of genus $g$ satisfies the functional equation
$$
\zeta _X(t) = \zeta _X \left( \frac{1}{qt} \right) q^{g-1} t^{2g-2}
$$
(cf.Proposition V.1.13 from \cite{St}).
This is not true for an arbitrary locally finite module $M$ over $\mathfrak{G} = {\rm Gal} ( \overline{\FF_q} / \FF_q)$.
Even if $M$ satisfies the Riemann Hypothesis Analogue with respect to $\PP^1 ( \overline{\FF_q})$ and $\zeta _M(t) = \frac{\prod\limits _{j=1} ^d (1 - q^{\lambda} e^{i \varphi _j} t)}{(1-t)(1-qt)}$ for some $\lambda \in \RR ^{>0}$, $\varphi _j \in [0, 2 \pi)$, the ratio
$$
\frac{\zeta _M(t)}{\zeta _M \left( \frac{1}{q^st} \right)} = C t^r \ \ \mbox{  with  } \ \ C \in \CC, \ \  r \in \ZZ
$$
has no poles and zeros outside the origin $0 \in \CC  $ if and only if $s=1$, $\lambda = \frac{1}{2}$,
$\zeta _M(t) = {\rm sign} (a_d)  \zeta _M \left( \frac{1}{qt} \right) q^{\frac{d}{2}-1} t^{d-2}$ or
$s=0$, $\lambda = d =1$, $\zeta _M(t) = \frac{1}{1-t}$.
More precisely,
\begin{align*}
\frac{\zeta _M(t)}{\zeta _M \left( \frac{1}{q^st} \right)} = \frac{\left[ \prod\limits _{j=1} ^d (1 - q^{\lambda} e^{i \varphi _j} t) \right] (1 - q^st)
(1 - q^{s-1} t)}{ \left[ \prod\limits _{j=1} ^d (1 - q^{s - \lambda} e^{- i \varphi _j} t) \right] (1-t)(1-qt)} a_d ^{-1} q^{sd - 2s +1} t^{d-2}
\end{align*}
 for the leading coefficient ${\rm LC} (P_M(t) )= a_d =   (-1)^d q^{  \lambda d} e^{  i \left( \sum\limits _{j=1} ^d \varphi _j \right)} \in \CC^*$ of the
 polynomial $\zeta$-quotient
 $P_M(t) = \frac{\zeta _M(t)}{\zeta _{\PP^1 ( \overline{\FF_q})} (t)} = \prod\limits _{j=1} ^d (1 - q^{\lambda} e^{i \varphi _j} t) $.
 Since $P_M(t) \in \ZZ[t] \subset \RR [t]$ is invariant under complex conjugation, the sets $\left \{ e^{i \varphi _j} \, \vert \, 1 \leq j \leq d \right \} = \left \{ e^{ - i \varphi _j} \, \vert \, 1 \leq j \leq d \right \}$ coincide, when counted with multiplicities and
\begin{align*}
\frac{\zeta _M(t)}{\zeta _M \left( \frac{1}{q^st} \right)} = \frac{\left[ \prod\limits _{j=1} ^d (1 - q^{\lambda} e^{i \varphi _j} t) \right] (1 - q^st)
(1 - q^{s-1} t)}{ \left[ \prod\limits _{j=1} ^d (1 - q^{s - \lambda} e^{  i \varphi _j} t) \right] (1-t)(1-qt)}  a_d^{-1} q^{sd - 2s +1} t^{d-2}.
\end{align*}
If $\frac{\zeta _M(t)}{\zeta _M \left( \frac{1}{q^st} \right)}$ has no zeros and poles outside $0 \in C$ then the sets
$$
\{ q^s, q^{s-1}, q^{\lambda} e^{i \varphi _j} \, \vert \, 1 \leq j \leq d \} = \{ q, 1, q^{s - \lambda} e^{i \varphi _j} \, \vert \, 1 \leq j \leq d \}
$$
coincide, when counted with multiplicities.
The assumption $\lambda > 0$ implies $q^s \neq q^{s - \lambda} e^{i \varphi _j}$ for $\forall 1 \leq j \leq d$ and specifies that $q^s \in \{ q, 1 \}$.
Therefore $s=1$ or $s=0$.

In the case of $s=1$, the coincidence $\{ q^{\lambda} e^{i \varphi _j} \, \vert \, 1 \leq j \leq d \} = \{ q^{1 - \lambda} e^{ i \varphi _j} \, \vert \, 1 \leq j \leq d \}$ of sets with multiplicities requires $\lambda = \frac{1}{2}$ and specifies that
$$
\zeta _M(t) =  {\rm sign} (a_d) \zeta _M \left( \frac{1}{qt} \right) q^{\frac{d}{2}-1} t^{d-2},
$$
according to $\left| a_d \right| = q^{\frac{d}{2}}$.

Suppose that $s=0$.
 Then the coincidence of sets
    $\{ 1, q^{-1}, q^{\lambda} e^{i \varphi _j} \, \vert \, 1 \leq j \leq d \} =
    \{ 1, q, q^{- \lambda} e^{i \varphi _j} \, \vert \, 1 \leq j \leq d \}$ with multiplicities leads to   $q^{-1} = q^{- \lambda} e^{i \varphi _j}$
     for some $1 \leq j_o \leq d$.
Thus, $\lambda =1$ and $\{ q e^{i \varphi _j} \, \vert \, 1 \leq j \leq d \} = \{ q \}$,  whereas $d=1$,
 $$
 \zeta _M(t) = \frac{1-qt}{(1-t)(1-qt)} = \frac{1}{1-t}
 $$
 and $M$ consists of a single $\mathfrak{G}$-fixed point.

The next corollary establishes that the Riemann Hypothesis Analogue with respect to the projective line $\PP^1 ( \overline{\FF_q})$ for an arbitrary locally finite $\mathfrak{G} = {\rm Gal} ( \overline{\FF_q} / \FF_q)$-module $M$ implies a functional equation for the polynomial $\zeta$-quotient
$P_M(t) = \frac{\zeta _M(t)}{\zeta _{\PP^1 ( \overline{\FF_q})} (t)} \in \ZZ[t]$.

\begin{corollary}    \label{RHAImpliesFE}
Let $M$ be an infinite locally finite module over $\mathfrak{G} = {\rm Gal} ( \overline{\FF_q} / \FF_q)$, which satisfies the Riemann Hypothesis Analogue with respect to $\PP^1 ( \overline{\FF_q})$.
Then the polynomial $\zeta$-quotient
$P_M(t) = \frac{\zeta _M(t)}{\zeta _{\PP^1 ( \overline{\FF_q})} (t)}  = \sum\limits _{j=0} ^d a_j t^{j} \in \ZZ[t]$ of $M$ satisfies the functional equation
$$
 P_M(t) = {\rm sign} (a_d) P_M \left( \frac{1}{q^{2 \lambda} t} \right) q^{\lambda d} t^d \ \ \mbox{   for } \ \ \lambda := \log _q \sqrt[d]{\left| a_d \right|}.
 $$
\end{corollary}

\begin{proof}

If $P_M(t) = \prod\limits _{j=1} ^d (1 - q ^{\lambda} e^{i \varphi _j} t)$ for some $\varphi _j \in [0, 2 \pi)$ then the leading coefficient
 ${\rm LC} ( PM(t)) = a_d = (-1)^d q^{\lambda d} e^{i \left( \sum\limits _{j=1} ^d \varphi _j \right) }    $, whereas 
 $$
 P_M \left( \frac{1}{q^{2 \lambda} t} \right) = \frac{a_d}{q^{2 \lambda d} t^d} \prod\limits _{j=1} ^d (1 - q^{\lambda} e^{-i \varphi _j} t).
 $$
The polynomial $P_M(t) \in \ZZ[t]$ has real coefficients and is invariant under the complex conjugation.
 Thus, the sets  $\left \{ e^{i \varphi _j} \, \vert \, 1 \leq j \leq d \right \} = \left \{ e^{ - i \varphi _j} \, \vert \, 1 \leq j \leq d \right \}$ coincide when counted with multiplicities   and    $P_M(t) = \prod\limits _{j=1} ^d (1 - q^{\lambda} e^{ - i \varphi _j} t)$.
That allows to express 
 $$
 P_M \left( \frac{1}{q^{2 \lambda} t} \right) = \frac{a_d}{q^{2 \lambda d}} P_M(t) t^{-d}.
 $$
Making use of $\left| a_d \right| = q^{\lambda d}$ and $a_d = {\rm sign} (a_d) \left| a_d \right|$, one concludes that 
$$
P_M \left( \frac{1}{q^{2 \lambda} t} \right) = \frac{{\rm sign} ( a_d)}{q^{\lambda d}} P_M(t) t^{-d}.
$$

\end{proof}

%% \newpage
%%%%%%%%%%%%%%%%%%%%%%%%%%%%%%%%%%%%%%%%%%%%%%%%%%%%%%%%%%%%%%%%%%%%%%%%%%%%%%%%%%%%%%%%%%%%%%%%%%%%%%%%%%%%%%%%%%%%%%%%%%%%%%%%%%%%%%%%%%

 \end{document}